\documentclass{amsart}
\usepackage[utf8]{inputenc}
\usepackage{amsfonts}
\usepackage{verbatim}
\usepackage{amsbsy}
\usepackage{amsmath}
\usepackage{amsthm}
\usepackage{amssymb}
\usepackage{graphicx}
\usepackage{caption}
\usepackage{subcaption}
\usepackage{skak}
\usepackage{pinlabel}
\usepackage[all]{xy}

\title{A Hennings TQFT Construction for Quasi-Hopf Algebras}
\author{Jennifer George}
\thanks{The author would like to thank T. Kerler for helpful comments and editorial advice, as well as J.S. Carter and M. Saito for encouraging discussions.}
\date{November 2013}

\makeatletter
\newcommand*\rel@kern[1]{\kern#1\dimexpr\macc@kerna}
\newcommand*\widebar[1]{%
  \begingroup
  \def\mathaccent##1##2{%
    \rel@kern{0.8}%
    \overline{\rel@kern{-0.8}\macc@nucleus\rel@kern{0.2}}%
    \rel@kern{-0.2}%
  }%
  \macc@depth\@ne
  \let\math@bgroup\@empty \let\math@egroup\macc@set@skewchar
  \mathsurround\z@ \frozen@everymath{\mathgroup\macc@group\relax}%
  \macc@set@skewchar\relax
  \let\mathaccentV\macc@nested@a
  \macc@nested@a\relax111{#1}%
  \endgroup
}
\makeatother

\newcommand{\inv}{^{-1}}
\newcommand{\opp}{^{\rm op}}
\newcommand{\dwg}{D^\omega[G]}
\newcommand{\bdry}{\partial}
\newcommand{\Rmat}{{\mathcal R}}

\newcommand{\id}{{\rm id}}
\newcommand{\Hom}{{\rm Hom}}
\newcommand{\tanglep}{TGL^{()}}
\newcommand{\deltaR}[1]{\Delta_R^{(#1)}}
\newcommand{\cal}{\mathcal}
\newcommand{\DTgl}{DTgl^{()}}

\newtheorem{theorem}{\bf Theorem}[section]
\newtheorem{definition}[theorem]{\bf Definition}
\newtheorem{notation}[theorem]{\bf Notation}
\newtheorem{lemma}[theorem]{\bf Lemma}

\theoremstyle{remark}
\newtheorem{remark}[theorem]{\bf Remark}

\numberwithin{equation}{section}

\begin{document}

\maketitle

\begin{abstract}
    We extend the construction of the Hennings TQFT for ribbon Hopf algebras to the case of ribbon quasi-Hopf algebras as defined by Drinfeld.   Calculations proceed in a similar fashion to the ordinary Hopf algebra case, but also require the handling of the non-trivial coassociator in the triple tensor product of the algebra as well as several special elements.  The main technical difficulties we encounter are representing tangle categories in the non-associative setting, and the definition and use of integrals and cointegrals in the non-coassociative case.  We therefore discuss the integral theory for quasi-Hopf algebras, using work of Hausser and Nill.  A motivating example for this work is the Dijkgraaf-Pasquier-Roche algebra which is believed to be related to the Dijkgraaf-Witten TQFT.
\end{abstract}


\section{Introduction}

Functorial invariants of 3-manifolds, more commonly referred to as Topological Quantum Field Theories, or TQFTs, have opened up a new view in low-dimensional topology and spurred much research since their axiomatic formulation by Atiyah in \cite{Atiyah}.  

Recall that a TQFT is a functor 
\begin{equation*}
    \nu : Cob_3 \to Vect,
\end{equation*}
where $Cob_3$ denotes the category of 2-framed cobordisms between standard surfaces, up to homeomorphism.  Numerous non-trivial constructions began to emerge around 1990, including those of Dijkgraaf-Witten, Reshetikhin-Turaev, Turaev-Viro, Witten, and others.  For some examples, see \cite{TV, DW, FQ, Ker1}. 

For a finite-dimensional, quasi-triangular, ribbon Hopf algebra $H$, we can construct a TQFT $\nu_H$ using tangle presentations of $Cob_3$ and a calculus of planar diagrams of links decorated with elements of $H$.  This construction, in the case of closed 3-manifolds, was introduced by Hennings in \cite{Hennings}, reformulated by Kauffman and Radford in \cite{KR}, and extended to a TQFT by Kerler in \cite{Ker1}.

The main result of this paper is an extension of this construction to the case where $H$ is a quasi-triangular ribbon quasi-Hopf algebra in the sense of Drinfeld \cite{Drin}.  The main complication with using a quasi-Hopf algebra is that the coassociativity condition is relaxed in the following way: there exists some element $\Phi \in H^{\otimes 3}$ so that
\begin{equation*}
    (\id \otimes \Delta(h))(\Delta(h)) = \Phi(\Delta \otimes \id)(\Delta(h))\Phi\inv
\end{equation*}
for every $h \in H$. We also require some additional normalization conditions such as $\lambda(\alpha v S\inv(\beta)) \lambda(\alpha v\inv S\inv(\beta)) = 1$ for a right cointegral $\lambda$, and that the monodromy element ${\mathcal M}$ is nondegenerate in the sense of Definition \ref{Mnondegen}.  See Definition \ref{assumptions} for specifics.

For technical reasons, we will use $Cob_3^\bullet$, the category of once-punctured cobordisms, instead of $Cob_3$.  See \cite{KL, CK}.  The main theorem will then be the following.
\begin{theorem}
	(Theorem \ref{thisisthemainidea})  Given a normalized unimodular quasi-Hopf algebra $H$, there exists a well-defined Hennings TQFT
	\begin{equation}
		\tilde{\nu_H} : Cob_3^\bullet \to Vect,
	\end{equation}
	extending the usual Hennings construction.
\end{theorem}

An important note is that all of our constructions will reduce to those in \cite{Ker1} in the case that the cocycle is trivial, and will further reduce to the case of \cite{KR} in the case that the cocycle is trivial and we work only with links.  Some algebraic elements of the construction were discussed in 1992 by Altschuler and Coste in \cite{AC}, but they do not give a full definition of the TQFT.  

The two main technical challenges we face in proving Theorem \ref{thisisthemainidea} are handling integrals and cointegrals in the quasi-associative situation, and developing a calculus on tangles that contains information about the coassociator.  To address the former difficulty, we have the following lemma in Section \ref{quasiInt}, extending results in \cite{BC, HN}. 
\begin{lemma}
    (Lemma \ref{cointHelper}) Let $\lambda \in H^\ast$ be a nonzero map.  Then $\lambda$ is a right cointegral on $H$ if and only if it satisfies the following property for any $h \in H$.
		\begin{equation}
			(\lambda \otimes \id)(q_L \Delta(h) p_L) = 1 \cdot \lambda(\alpha h S\inv(\beta))
		\end{equation}
	Here $q_L, p_L \in H \otimes H$ are identities related to the action of the antipode.
\end{lemma}
To address the latter difficulty, we add subtleties to the system of functors used to describe the original Hennings TQFT $\nu_H$.  

This extended Hennings TQFT is important in its own right, but its study was motivated by the following connection.  Dijkgraaf, Pasquier and Roche defined in \cite{DPR} a quasi-Hopf algebra satisfying the conditions of Definition \ref{assumptions}.  This quasi-Hopf algebra, denoted $\dwg$, is defined using a 3-cocycle $\omega \in Z^3(G, \mathbb{C}^\times)$ to relax the coassociativity condition.  The Pentagon Axiom corresponds exactly to the group cocycle condition.  It is generally believed that this quasi-Hopf algebra is related to the TQFT of Dijkgraaf and Witten \cite{DW}, or its reformulations by Wakui \cite{Wakui} or Freed and Quinn \cite{FQ}, as this TQFT is defined using the same 3-cocycle $\omega$.

A precise correspondence between the two notions is difficult to formulate without a TQFT defined from $\dwg$.  The Hennings TQFT provides this equivalence in the case where the cocycle is trivial or $\omega = 1$ \cite{G1}.  Using the formulation of the Hennings TQFT for quasi-Hopf algebras as given in this paper, we then expect the equivalence to hold in the more general case.  That is, we can make a precise conjecture about the correspondence: the Hennings TQFT applied to the quasi-Hopf algebra $\dwg$ is equivalent to the Dijkgraaf-Witten TQFT.

\subsection{Overview}

We provide an overview of the structure of this paper.  For organizational purposes, we will refer to the diagram in Figure \ref{theplan}.  Our main goal is to define the composition $\tilde{\nu_H} : Cob_3^\bullet \to Vect$, or the extension of the Hennings TQFT, and prove that it is well-defined.
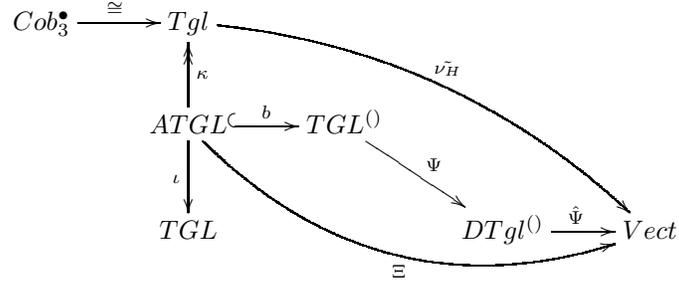
\begin{figure}
    $$  \xymatrix{Cob_3^\bullet \ar[r]^{\cong} & Tgl \ar@/^1.5pc/[ddrrr]^{\tilde{\nu_H}} & & & \\
                & ATGL \ar@/_2.5pc/[drrr]_{\Xi} \ar@{>>}[u]_{\kappa} \ar@{^{(}->}[r]^b \ar[d]_\iota & \tanglep \ar[dr]^{\Psi} & & \\
                & TGL & & DTgl^{()} \ar[r]^{\hat{\Psi}} & Vect }$$
    \caption{The Desired TQFT}
    \label{theplan}
\end{figure}

In Section 2 we review the basics for quasi-Hopf algebras and estabish our notation.  We also discuss in Section \ref{quasiInt} the integral theory for quasi-Hopf algebras.

In Section 3, we define the categories in Figure \ref{theplan}.  The category $Cob_3^\bullet$ refers to a 3-dimensional cobordism cateogry where the standard surfaces have one puncture, and this category is treated in Section \ref{CobCat}.  We use $TGL$ to refer to the usual tangle category, and then the tangle categories $Tgl$ and $ATGL$ have additional restrictions.  We also describe categories $\tanglep$ and $A\tanglep$ where we add brackets to the tangles, where by ``bracketed'' we mean decorated with parenthesis.  Finally, we have a category $DTgl^{()}$ of tangles where we remove over- and under-crossing information, and may decorate our tangles with beads labeled by elements in a quasi-Hopf algebra. Both functors labeled $\iota$ and $b$ are inclusion of categories.  These tangle categories are treated in Section \ref{tangleCatDefs}.

Finally, in Section \ref{newhenningsch}, we discuss the functors $\Psi$ and $\hat{\Psi}$.  Their definitions are found in Section \ref{quasi-Henningsdef}, where we also prove that $\hat{\Psi}$ is well-defined.  In Section \ref{well-definedness} we prove that $\Psi$ is well-defined, and in Section \ref{diagramworks} we prove that the full composition, or the TQFT $\tilde{\nu_H}$, is well-defined.  This last proof will be the main result of the paper.

\section{Preliminaries} \label{prelim}

\subsection{Quasi-Hopf Algebras}\label{qhasection}

Recall \cite{AC, Ka} that a quasi-Hopf algebra is a Hopf algebra in which we have weakened the coassociativity condition by what we call a coassociator $\Phi$.  Quasi-Hopf algebras were originally defined by Drinfeld in \cite{Drin}.  More specifically, we have the following.

Let $H$ be a finite-dimensional vector space over a field ${\bf k}$ of characteristic zero.  Suppose this $H$ is equipped with an associative multiplication $\mu$ and a unit map $\eta$.

\begin{definition}
    \cite{Drin} The vector space $H$ is a quasi-bialgebra if we have algebra homomorphisms $\Delta : H \to H \otimes H$ and $\epsilon : H \to {\bf k}$, as well as an invertible element $\Phi \in H^{\otimes 3}$ so that the following properties are satisfied for any $h \in H$.
    \begin{align}
        & (\id \otimes \Delta)(\Delta(h)) = \Phi (\Delta \otimes \id)(\Delta(h)) \Phi \inv \nonumber\\
        & (\id \otimes \id \otimes \Delta)(\Phi) (\Delta \otimes \id \otimes \id)(\Phi) = (1\otimes \Phi) (\id \otimes \Delta \otimes \id) (\Phi \otimes 1) \label{pentagonaxiom} \\
        & (\epsilon \otimes \id)(\Delta(h)) = h  = (\id \otimes \epsilon)(\Delta(h)) \nonumber \\
        & (\id \otimes \epsilon \otimes \id) (\Phi) = 1 \otimes 1 \nonumber
    \end{align}
\end{definition}

We denote the coproduct of an element $h \in H$ by 
\begin{equation}
    \Delta(h) = \sum_{(h)} h' \otimes h''.
\end{equation}
This notation is due to Sweedler, and as usual we will suppress the summation index.  Since the coproduct is not coassociative, we must take care to distinguish higher coproducts as follows.
\begin{equation} \label{coprodex}
    \sum h' \otimes (h'')' \otimes (h'')'' \neq \sum (h')' \otimes (h')'' \otimes h''
\end{equation}
For the purposes of this paper, the notation in \eqref{coprodex} will suffice.  For more complicated applications of higher coproducts, the reader is referred to the notation in \cite{HN}.

We denote the coassociator
\begin{equation}
    \Phi = \sum_i X_i \otimes Y_i \otimes Z_i,
\end{equation}
where again we will suppress the summation index.  Its inverse we denote
\begin{equation}
    \Phi\inv = \sum \bar{X_i} \otimes \bar{Y_i} \otimes \bar{Z_i}.
\end{equation}

\begin{definition} \label{qha}
    A quasi-bialgebra is a quasi-Hopf algebra if there exists an antiautomorphism $S: H \to H$, called the antipode, and elements $\alpha, \beta \in H$ so that the following hold for every $h \in H$.
    \begin{align}
        & \sum S(h') \alpha h'' = \epsilon(h) \alpha \nonumber\\
        & \sum h' \beta S(h'') = \epsilon(h) \beta \nonumber\\
        & \sum X_i \beta S(Y_i) \alpha Z_i = 1  \nonumber\\
        & \sum S(\bar{X_i}) \alpha \bar{Y_i} \beta S(\bar{Z_i}) = 1
    \end{align}
\end{definition}

Two consequences of Definition \ref{qha} are the following.
\begin{align}
    & \epsilon \circ S = \epsilon \nonumber \\
    & \epsilon(\alpha \beta) = \epsilon(\alpha) \epsilon(\beta) = 1 
    \label{epsiloncondition}
\end{align}
Equation \eqref{epsiloncondition} implies that we may assume $\epsilon(\alpha) = \epsilon(\beta) = 1$ by rescaling $\alpha$ and $\beta$ if necessary.  We make this assumption for the remainder of this paper.

\begin{notation}
    Let $\sigma$ denote a permutation of $\{1, 2, 3\}$.  Then we use $\Phi_{\sigma(1)\sigma(2)\sigma(3)}$ to denote $\sigma(\Phi)$, using the natural $S_3$ action on $H^{\otimes 3}$.  For example, 
    \begin{align*}
        \Phi_{312} = \sum Y_i \otimes Z_i \otimes X_i.
    \end{align*}
\end{notation}

\begin{definition}
    \cite{Drin} A quasi-Hopf algebra $H$ is quasitriangular if there exists an invertible element $\Rmat \in H \otimes H$ so that the following hold for every $h \in H$.  Let $\Rmat = \sum s_i \otimes t_i$, and let $\Rmat_{ab} \in H^{\otimes 3}$ denote the two factors of $\Rmat$ in positions $a$ and $b$ and the identity in the other factor.  For example, $\Rmat_{13} = \sum s_i \otimes 1 \otimes t_i$.  Also let $\Delta\opp$ denote the coproduct composed with a flip map, exchanging the two factors.
    \begin{align}
        & \Delta\opp(h) = \Rmat \Delta(h) \Rmat \inv \nonumber \\
        & (\Delta \otimes \id)(\Rmat) = \Phi_{312} \Rmat_{13} \Phi\inv_{132} \Rmat_{23} \Phi \label{hexaxiom1} \\
        & (\id \otimes \Delta)(\Rmat) = \Phi\inv_{231} \Rmat_{13} \Phi_{213} \Rmat_{12} \Phi\inv \label{hexaxiom2}
    \end{align}
\end{definition}

Condition \eqref{pentagonaxiom} is called the Pentagon Axiom, while Conditions \eqref{hexaxiom1} and \eqref{hexaxiom2} are called the Hexagon Axiom.  The Hexagon Axiom is equivalent to the quasi-Yang Baxter equation:
\begin{equation}\label{qyb}
    \Rmat_{12} \Phi_{312} \Rmat_{13} \Phi\inv_{132} \Rmat_{23} \Phi = \Phi_{321} \Rmat_{23} \Phi\inv_{231} \Rmat_{13}\Phi_{213} \Rmat_{12}.
\end{equation}

We denote the $\Rmat$-matrix as
\begin{equation}
    \Rmat = \sum s_i \otimes t_i,
\end{equation}
and its inverse as 
\begin{equation}
    \Rmat\inv = \sum \bar{s_i} \otimes \bar{t_i}.
\end{equation}

From \cite{AC, Ka}, one fact about quasitriangular quasi-Hopf algebras is that they contain an invertible element $u$, defined as follows.
\begin{align}
    u = \sum S(\bar{Y_i} \beta S(\bar{Z_i}))S(t_j) \alpha s_j \bar{X_i}
\end{align}

The element $u$ has various properties, including the following.  See \cite{AC}.
\begin{align}
    &S^2(u) = u \nonumber\\
    &uS(u) = S(u)u \text{ is central} \nonumber\\
    &\sum S(t_i) \alpha a_i = S(\alpha)u = S(u)u \sum S(\bar{s_i}) \alpha \bar{t_i}
\end{align}

We would like to work with ribbon quasi-Hopf algebras, so we will need to define special elements $G$ and $v$, following \cite{AC}.  In order to make these definitions, however, we first need some additional elements.

To simplify the notation, let $(\Phi \otimes 1)(\Delta \otimes \id \otimes \id)(\Phi \inv) =\sum_i A_i \otimes B_i \otimes C_i \otimes D_i$.  Then  define $\gamma = \sum_i S(B_i) \alpha C_i \otimes S(A_i)\alpha D_i$.  Similarly, to simplify the notation, let $(\Delta \otimes \id \otimes \id)(\Phi) (\Phi \inv \otimes 1) = \sum K_i \otimes L_i \otimes M_i \otimes N_i$, and then define $\delta = \sum K_i \beta S(N_i) \otimes L_i \beta S(M_i)$.  Using these definitions, we define an element $f$ which helps simplify some later algebra.  Let $f = \sum_i (S \otimes S)(\Delta^{\rm op}(\bar{X_i})) \cdot \gamma \cdot \Delta(\bar{Y_i} \beta S(\bar{Z_i}))$.

These elements enjoy the following relations.
\begin{lemma} \label{propsofgdf}
	\cite[(2.16), (2.17), (2.18)]{AC} The elements $\gamma$, $\delta$, and $f$ satisfy the following, for any $h \in H$.
	\begin{align}
		&f \Delta(h) f\inv = (S \otimes S)(\Delta^{\rm op}(S\inv(h))) \nonumber\\
		&\gamma = f \Delta(\alpha) \nonumber\\
		&\delta = \Delta(\beta) f\inv \label{deltafprop}
	\end{align}
\end{lemma}

\begin{definition} \label{vprops}
	\cite[Section 4.1]{AC} We call a quasitriangular quasi-Hopf algebra a ribbon quasi-Hopf algebra if there exists an invertible element $v \in H$ so that $v^2 = uS(u)$, $S(v) = v$, $\epsilon(v) = 1$, and the following holds.
	\begin{align}
        \Delta(uv\inv) = f\inv((S\otimes S)(f_{21}))(uv\inv \otimes uv\inv) \label{vcond4}
	\end{align}
	Here, we let $f_{21} = \tau(f)$ for the flip map $\tau$.
\end{definition}

We now denote $G = uv\inv$, and note that $G$ is not grouplike in a quasi-Hopf algebra, but instead satisfies condition \eqref{vcond4}.  However, if we reduce to the strict case $\Phi = 1 \otimes 1 \otimes 1$, we have $f = 1$ and hence $G$ is grouplike.  The element $G$ will also simplify some later calculations, and is related to the First Reidemister Move, as we will see in Chapter \ref{newhenningsch}.

\begin{lemma} \label{Gprops}
	The following properties hold for the special element $G = uv\inv$, and for all $h \in H$.
	\begin{align}
		&S(G) =  G\inv \nonumber \\ 
		&S^2(h) =  GhG\inv \nonumber\\
		&uG = Gu \nonumber \\
		&S(u) = G\inv u G\inv \label{Gprop4}
\end{align}
\end{lemma}
These properties are discussed in \cite{KR} in the ordinary Hopf case; their proof is similar and straightforward using the definition of $G$.

We define the element ${\mathcal M}$, called the monodromy element, as follows.  Let $\Rmat\opp = \sum t_i \otimes s_i$.
\begin{equation} \label{Mdef}
	{\mathcal M} = \Rmat^{\rm op} \Rmat = \sum t_j s_i \otimes s_j t_i
\end{equation}

\begin{definition} \label{Mnondegen}
	\cite{CK} We say that ${\mathcal M}$ is nondegenerate if
	\begin{equation}
		\bar{\mathcal M} : H^\ast \to H : l \mapsto (\id \otimes l)({\mathcal M})
	\end{equation}
	is an isomorphism.
\end{definition}
For the remainder of this paper, we assume that $\mathcal{M}$ is non-degenrate.

Finally, we have four more special elements in $H \otimes H$, given in \cite{HN}, which will help simplify later calculations.  
\begin{align}
	&p_L = \sum \tilde{p}^1 \otimes \tilde{p}^2 = \sum Y_i S\inv(X_i\beta) \otimes Z_i \label{pldef}\\
	&p_R = \sum p^1 \otimes p^2 = \sum \bar{X_i} \otimes \bar{Y_i} \beta S(\bar{Z_i}) \label{prdef} \\
	&q_L = \sum \tilde{q}^1 \otimes \tilde{q}^2 = \sum S(\bar{X_i}) \alpha \bar{Y_i} \otimes \bar{Z_i} \label{qldef} \\
	&q_R = \sum q^1 \otimes q^2 = \sum X_i \otimes S\inv(\alpha Z_i) Y_i \label{qrdef}
\end{align}
We will sometimes use the shorthand notation $p_L = \tilde{p}$, $q_L = \tilde{q}$.

\subsection{Integrals and Cointegrals}\label{quasiInt}

Even though a non-coassociative $H$ means that the dual $H^\ast$ is not associative, hence not an algebra, cointegrals still exist by considering $H^\ast$ as a left quasi-Hopf $H$-bimodule.  Integrals and cointegrals also behave generally in the way we expect from the ordinary Hopf algebra case.  For the definitions and properties of integrals and cointegrals in general quasi-Hopf algebras, see \cite{HN, BC}. We will outline the properties needed for this paper and simplify the general results to the case of unimodular quasi-Hopf algebras.

The results for the quasi-Hopf case are much the same as the ordinary Hopf case.  A quasi-Hopf algebra $H$ is unimodular if every left integral is also a right integral.  We use $\Lambda \in H$ to denote an integral.  Also, the spaces of integrals and cointegrals are each one-dimensional, see \cite{HN}.

\begin{definition}
	\cite{BC} An element $\lambda \in H^\ast$ is called a right cointegral of a unimodular finite-dimensional quasi-Hopf algebra $H$ if and only if
	\begin{align}
		(\lambda \otimes \id)((S \otimes S)(\tilde{p}_{21}) f \Delta(h) (S\inv \otimes S\inv)(f\inv_{21})(S\inv \otimes S\inv)(\tilde{q}_{21})) = \lambda(h) \label{rightcointdef}
	\end{align}
	for every $h\in H$, where $\tilde{p} = p_L$ and $\tilde{q} = q_L$.
\end{definition}

One may also define left cointegrals, but we will only need right cointegrals for our purposes.  To make the definition of right cointegrals more applicable, we tailor the definition to our needs.  Recall from Section \ref{qhasection} that we have assumed both $\epsilon(\alpha) = \epsilon(\beta) = 1$ and that $\mathcal{M}$ is non-degenerate.  The next lemma is originally in \cite{BC}, but we remark that the authors applied the maps of $H^{\rm cop}$ to a previous result for left cointegrals, and mistakenly took the $\beta$ element to be $S\inv (\alpha)$.

\begin{lemma}\label{BCcointHelper}
	\cite{BC} Let $\lambda \in H^\ast$ be nonzero, and let $\Lambda$ be any left integral in $H$.  Then $\lambda$ is a right cointegral on $H$ if and only if one of the equivalent relations below is satisfied.
		\begin{align}
			&(\lambda \otimes \id)(q_L \Delta( \Lambda) p_L)  = \epsilon(\beta) \lambda(\Lambda) \cdot 1 \label{BCcond1} \\
			&(\lambda \otimes \id)(\Delta(\Lambda) p_L) = \epsilon(\beta)\lambda(\Lambda) S\inv(\beta) \label{BCcond2} \\
			&(\lambda(h\Lambda'\tilde{p}^1)) \Lambda''\tilde{p}^2 = \epsilon(\beta)\lambda(\Lambda)S\inv(h\beta) \hspace{0.2in} \text{ for all } h \in H \label{BCcond3}
		\end{align}
	where $ p_L =\sum  \tilde{p}^1 \otimes \tilde{p}^2$.
\end{lemma}

\begin{lemma}\label{cointHelper}
	Let $\lambda \neq 0 \in H^\ast$.  Then $\lambda$ is a right cointegral on $H$ if and only if it satisfies the following property for any $h \in H$.
		\begin{equation}
			(\lambda \otimes \id)(q_L \Delta(h) p_L) = 1 \cdot \lambda(\alpha h S\inv(\beta)) \label{lemmaprop}
		\end{equation}
\end{lemma}

\begin{proof}
	First, suppose $\lambda$ is a nonzero map satisfying \eqref{lemmaprop}.  Let $\Lambda$ be a left integral on $H$.  Since $H$ is unimodular, $\Lambda$ is also a right integral.  We have the following.
	\begin{align}
		(\lambda \otimes \id)(q_L \Delta(\Lambda) p_L) &= 1 \cdot \lambda(\alpha \Lambda S\inv(\beta)) \nonumber\\
		&= 1 \cdot \lambda(\epsilon(\alpha) \Lambda S\inv(\beta)) \nonumber \\
		&= \epsilon(\alpha) \lambda(\Lambda S\inv(\beta))  \nonumber\\
		&= \epsilon(\alpha) \epsilon(S\inv(\beta)) \lambda(\Lambda)  \nonumber\\
		&= 1 \cdot \epsilon(\beta) \lambda(\Lambda) \nonumber
	\end{align}
	Note that the last line follows as we have assumed $\epsilon(\alpha) = 1$ and we know that $\epsilon \circ S = \epsilon$ from properties of a quasi-Hopf algebra.
	
	By \eqref{BCcond1} in Lemma \ref{BCcointHelper}, $\lambda$ is a right cointegral.
	
	To prove the converse, let $\lambda$ be a right cointegral.  We show that \eqref{lemmaprop} is satisfied.  We substitute $\alpha h S\inv(\beta)$ in the definition of $\lambda$ given in \eqref{rightcointdef}.  This gives us the following equality.
	\begin{align}
		(\lambda \otimes \id)((S \otimes S)(\tilde{p}_{21}) f \Delta(\alpha h S\inv(\beta))& (S\inv \otimes S\inv)(f\inv_{21})(S\inv \otimes S\inv)(\tilde{q}_{21}))\nonumber \\
		& = \lambda(\alpha h S\inv(\beta)) \label{rcoint1}
	\end{align}
	Now, we know that $f \Delta(\alpha) = \gamma$, and we can compute that
	\begin{align*}
		(S \otimes S)(\tilde{p}_{21}) \gamma = \tilde{q} = q_L.
	\end{align*}
	This allows us to simplify \eqref{rcoint1} to the following.
	\begin{align}
		(\lambda \otimes \id)(q_L \Delta(h) \Delta(S\inv(\beta)) &(S\inv \otimes S\inv)(f\inv_{21})(S\inv \otimes S\inv)(\tilde{q}_{21})) \nonumber \\
		&= \lambda(\alpha h S\inv(\beta)) \label{rcoint2}
	\end{align}
	Next, we use \eqref{deltafprop} to see that $\Delta(S\inv(\beta)) (S\inv \otimes S\inv)(f\inv_{21}) = (S\inv \otimes S\inv)(\delta_{21})$, and then compute
	\begin{align}
		(S\inv \otimes S\inv)(\delta_{21}) (S\inv \otimes S\inv)(\tilde{q}_{21}) = \tilde{p} = p_L.
	\end{align}	
	Finally, then, we simplify \eqref{rcoint2} into the desired result.
	\begin{align*}
		(\lambda \otimes \id)(q_L \Delta(h) p_L) = \lambda(\alpha h S\inv(\beta))
	\end{align*}\end{proof}


Since $H$ is unimodular, Lemma 5.1 in  \cite{HN} implies that for all right cointegrals $\lambda$ and for all $a,b \in H$, we have
\begin{align} 
    \lambda(ab) = \lambda(S^2(b)a).\label{cointegralcommute}
\end{align}
We then have the following.
\begin{lemma}
    For $\lambda$ a right cointegral, the following holds for all $h \in H$.
    \begin{align*}
        \lambda(S(h)) = \lambda(G^2h)
    \end{align*}
\end{lemma}
In the ordinary Hopf case, this result follows from \cite{Rad} and the existence of a non-degenerate trace.  In the quasi-Hopf case, \cite{HN} gives that the same notion of a trace map is non-degenerate, hence the result still holds.

\section{Definitions of Cobordism and Tangle Categories}

Recall the diagram of Figure \ref{theplan}.  Our goal in this section is to define the necessary categories in the diagram.  For further details on these categories, see \cite{KL, CK}.  

We begin with $Cob_3^\bullet$. 

\subsection{Cobordism Category $Cob_3^\bullet$}\label{CobCat}

 Let $\Sigma_n^\bullet$ denote a connected, compact, oriented, genus $n$ surface with one boundary component isomorphic to the circle $S^1$. We associate to each surface $\Sigma_n^\bullet$ an orientation-preserving homeomorphism $\bdry \Sigma_n^\bullet \to S^1$.  Let $-\Sigma$ denote $\Sigma$ with opposite orientation.

    Let ${\mathcal C}_{mn}$ denote the stratified surface obtained by sewing a cylinder $C = S^1 \times [0,1]$ between $-\Sigma_m^\bullet$ and $\Sigma_n^\bullet$ using the two boundary homomorphisms.  That is, 
	\begin{equation}
		{\mathcal C}_{mn} = (- \Sigma_m^\bullet) \coprod_{\bdry \Sigma_m^\bullet \sim (S^1 \times 0)} (S^1 \times [0,1]) \coprod_{\bdry \Sigma_n^\bullet \sim (S^1 \times 1)} (\Sigma_n^\bullet).
	\end{equation}
Using the isomorphisms $\bdry \Sigma_m^\bullet \simeq S^1$, $\bdry \Sigma_n^\bullet \simeq S^1$ and the fact that $\bdry C \simeq -S^1 \sqcup S^1$, the combined surface ${\mathcal C}_{mn}$ admits an orientation compatible with its pieces $-\Sigma_m^\bullet$ and $\Sigma_n^\bullet$.

	\begin{figure}[ht]
		\centering
		\labellist
		\pinlabel $\Sigma^\bullet_m$ at 30 340
		\pinlabel $\Sigma_n^\bullet$ at 220 20
		\pinlabel $M$ at 340 160
		\pinlabel $C$ at 530 80
		\endlabellist
		\includegraphics[scale=0.5]{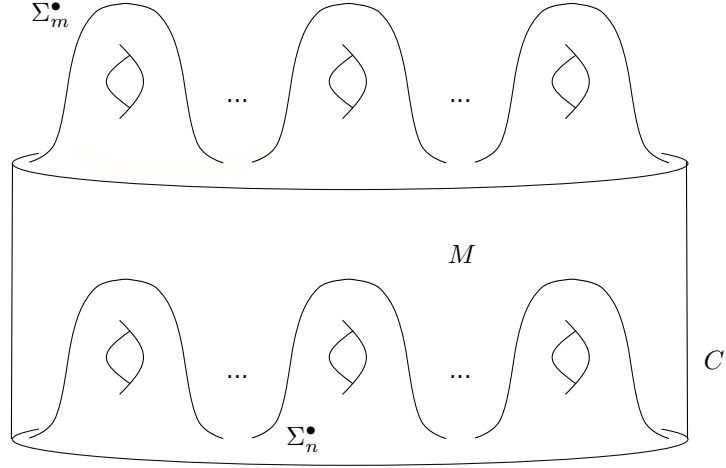}
		\caption{A Cobordism $M: \Sigma_m^\bullet \to \Sigma_n^\bullet$}
		\label{cobC}
	\end{figure}

    A cobordism between surfaces $\Sigma_m^\bullet$ and $\Sigma_n^\bullet$ is a compact oriented 3-manifold $M$ with corners together with a strata- and orientation-preserving homeomorphism $\xi : {\mathcal C}_{mn} \to \bdry M$.  See Figure \ref{cobC} for an example.  We will use the notation $M : \Sigma_m^\bullet \to \Sigma_n^\bullet$ to denote a cobordism ${\bf M}$ between $-\Sigma_m^\bullet$ and $\Sigma_n^\bullet$.  We say that two cobordisms ${\bf M}$ and ${\bf M'}$ are equivalent if $m = m'$, $n=n'$, and there is a homeomorphism $\eta : M \to M'$ so that $\eta \circ \xi = \xi'$.

	Now we define the category of framed cobordisms $Cob_3^\bullet$ in dimension $2 + 1$ as follows.  The objects are the standard surfaces $\Sigma_n^\bullet$, with one surface for each $n \in \mathbb{N}$.  The morphisms are the equivalence classes of cobordisms $[M, \xi]$, together with a 2-framing of $M$, or equivalently the signature of a 4-manifold bounding a standard closure of $M$. Composition of morphisms is defined by glueing cobordisms along a common boundary component and rescaling.  The tensor product is given by disjoint union.  We will use ${\bf M}$ to denote the class $[M, \xi]$.

\subsection{Tangle Categories $Tgl$ and $\tanglep$}\label{tangleCatDefs}

Continuing with our treatement of the categories in Figure \ref{theplan}, we begin with our basic tangle category.  Then, we add additional restrictions on the category and describe the relationships amongst the categories.  Our first category is the usual tangle category, with which most readers will be familiar.

\begin{definition} \label{tangledef}
    Let $n$ and $m$ be integers.  A framed tangle $T : m \to n$ is a diagram of circles and intervals in $\mathbb{R}^2$ with $m$ strands attached to the top of the tangle and $n$ strands attached to the bottom of the diagram.  Such diagrams are generated by the maps $\cap : 0\to 2$, $\cup : 2 \to 0$,  $c : 2 \to 2$ and $c\inv : 2 \to 2$, as well as a straight strand labeled $\id : 1 \to 1$. 
\end{definition}
The generating maps are pictured in Figure \ref{Tglgens}.  We may compose tangles by stacking, where $S \circ T$ is given by stacking $S$ on top of $T$. 

\begin{remark}
	We may also define tangles as framed embeddings of a union of circles and intervals in $\mathbb{R}^2 \times [0,1]$.  It is well-known \cite{Yetter} that this definition is equivalent to Definition \ref{tangledef}.  Also, instead of considering embeddings in $\mathbb{R}^2 \times [0,1]$, we may take a generic immersion of a union of circles and intervals in $\mathbb{R} \times [0,1]$ with over- and under-crossing information at double points.  This second definition is mapped to the first in this remark by pushing strands off each other at double points using the over- or under-crossing information and the blackboard framing.
\end{remark}

\begin{definition}
	The category $TGL$ is defined as follows.  Its objects are given by non-negative integers.  The morphisms are the tangles of Definition \ref{tangledef}.  Equivalence on diagrams is generated by isotopies in the plane, as well as Moves I, II, III, IV, and R.  These moves are pictured in Figure \ref{movesTgl}.  The tensor product is given by juxtaposition, and the composition by stacking.
\end{definition}
Alternatively, we can think of the morphisms in $\Hom(m, n)$ as equivalence classes of framed tangles with $m$ top end points and $n$ bottom end points.

\begin{figure}[ht]
    \centering
	\labellist
	\pinlabel $\cap$ at 42 -10
	\pinlabel $\cup$ at 155 -10 
	\pinlabel $c$ at 255 -10
	\pinlabel $c\inv$ at 345 -10 
	\pinlabel $\id$ at 410 -10
	\endlabellist
	\includegraphics[scale=0.7]{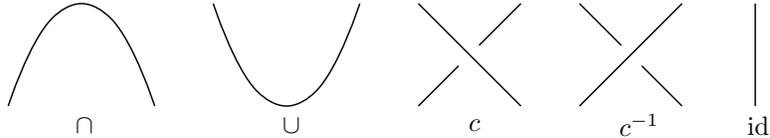}
	\caption{Generators for $Tgl$}
	\label{Tglgens}
\end{figure}

\begin{remark}
	We read our tangle diagrams from the top to the bottom.  That is, the source of the tangle is at the top of the diagram, and the target at the bottom.
\end{remark}

For the usual Hennings TQFT, we first restrict to a category $ATGL$ of what we will refer to as ``even'' tangles.  The objects are even integers, and the morphisms are tangles in which the endpoints are connected in pairs.  That is, at the pair $(j^-, j^+)$, or the $2j-1$-st and $2j$-th points, we either have a single component of the tangle connecting these endpoints, or a pair of components connecting these endpoints to the pair of endpoints $(i^-, i^+)$ at the opposite end of the tangle.  For an example, see Figure \ref{AdmissibleEx}.  There is an obvious inclusion $\iota : ATGL \to TGL$. See Figure \ref{theplan}.
\begin{figure}[ht]
    \centering
    \labellist
    \pinlabel $1^-$ at 20 190
    \pinlabel $1^+$ at 50 190
    \pinlabel $2^-$ at 100 190
    \pinlabel $2^+$ at 130 190
    \pinlabel $3^-$ at 170 190
    \pinlabel $3^+$ at 200 190
    \pinlabel $1^-$ at 40 -10
    \pinlabel $1^+$ at 70 -10
    \pinlabel $2^-$ at 100 -10
    \pinlabel $2^+$ at 130 -10
    \endlabellist
    \includegraphics[scale=0.6]{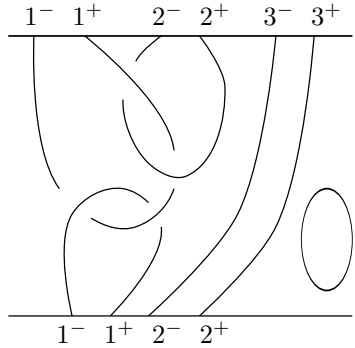}
    \caption{An Element of $ATGL$}
    \label{AdmissibleEx}
\end{figure}

We then quotient $ATGL$ by three extra equivalence relations: the $\sigma$-move, a modified First Kirby Move, and the Second Kirby move.  The $\sigma$-move is pictured in Figure \ref{movesTgl}.  We modify the First Kirby Move to state that we may remove of an isolated Hopf link in which one component has zero framing.  We make this modification so that we may replace the Second Kirby Move by the Fenn-Rourke Move.  These two moves are known to be equivalent in the case of links, and we discuss that they are still equivalent for tangles.  First, we make the relevant definitions.
\begin{definition}
    \cite[M.1]{FR}  Let $L$ be a link, and $C$ an unknotted component of $L$ labeled $\pm 1$.  Let $N$ be a tubular neighborhood of $C$ and $D$ a spanning disc with boundary a longitude of $\bdry N$. Cut $S^3 - N$ along $D$.  If $C$ is labeled $+1$, twist one side of $D$ in a left-handed anticlockwise manner through $2\pi$ and reglue.  If $C$ is labeled $-1$, take the opposite twist and reglue.  The component $C$ is now unlinked from the rest of the diagram.  This procedure is called the Fenn-Rourke Move.
	
	If $C''$ is another component of $L$ with framing $n''$, after performing the Fenn-Rourke Move, the framing on $C''$ will be given by
	\begin{align*}
		n'' \pm l(C, C'')^2,
	\end{align*}
	where $l$ denotes the linking number and the plus or minus depends on whether $C$ has framing plus or minus $1$.
\end{definition}

	The effect of the Fenn-Rourke move is the following.  If $n$ vertical strands pass through the component $C$, these $n$ strands receive a full $2\pi$-twist, and $C$ is unlinked from these strands.

Theorems 2, 3, 6, and 8 in \cite{FR} give proofs in various situations that the Fenn-Rourke Move is equivalent to the Second Kirby Move.  Since our manifold is not closed and compact, we need to also modify the First Kirby Move - instead of addition or removal of isolated closed components, we must use the Modified First Kirby Move - addition or removal of an isolated Hopf link in which one component has zero framing.

\begin{remark}
	All of the statements and results thus far have been made only for links.  We would like to extend these definitions to tangles in the following manner.  Tangles have both open components, isomorphic to intervals, and closed components.  For both the Second Kirby Move and the Fenn-Rourke Move, we restrict the component $C$ so that it must be a closed component of our tangle.  The other component $C''$ may be either open or closed.
	
	The proofs in \cite{FR} of the equivalence of these two moves are done locally, and in fact the graphics may be taken as tangles instead of links.  Hence, the equivalence of the two moves still holds in our tangle category.
\end{remark}

The category resulting from adding these three moves will be called the ``admissible tangle category'' and denoted $Tgl$.  For more information on the $\sigma$-move, see \cite{Ker2}.  We then have the following theorem.
\begin{theorem}\label{isomorphiccategorieslemma}
    \cite{Ker2} The category $Tgl$ of admissible tangles is isomorphic to the category $Cob_3^\bullet$ of cobordisms of once-punctured surfaces.
\end{theorem}

\begin{figure}[ht]
	\centering
	\labellist
	\pinlabel {Move I} at 120 470
	\pinlabel {Move II} at 420 470
	\pinlabel {Move III} at 115 220
	\pinlabel {Move IV} at 420 220
	\pinlabel {Move R} at 75 -10
	\pinlabel {$\sigma$-Move} at 430 -10 
	\endlabellist
	\includegraphics[scale=0.7]{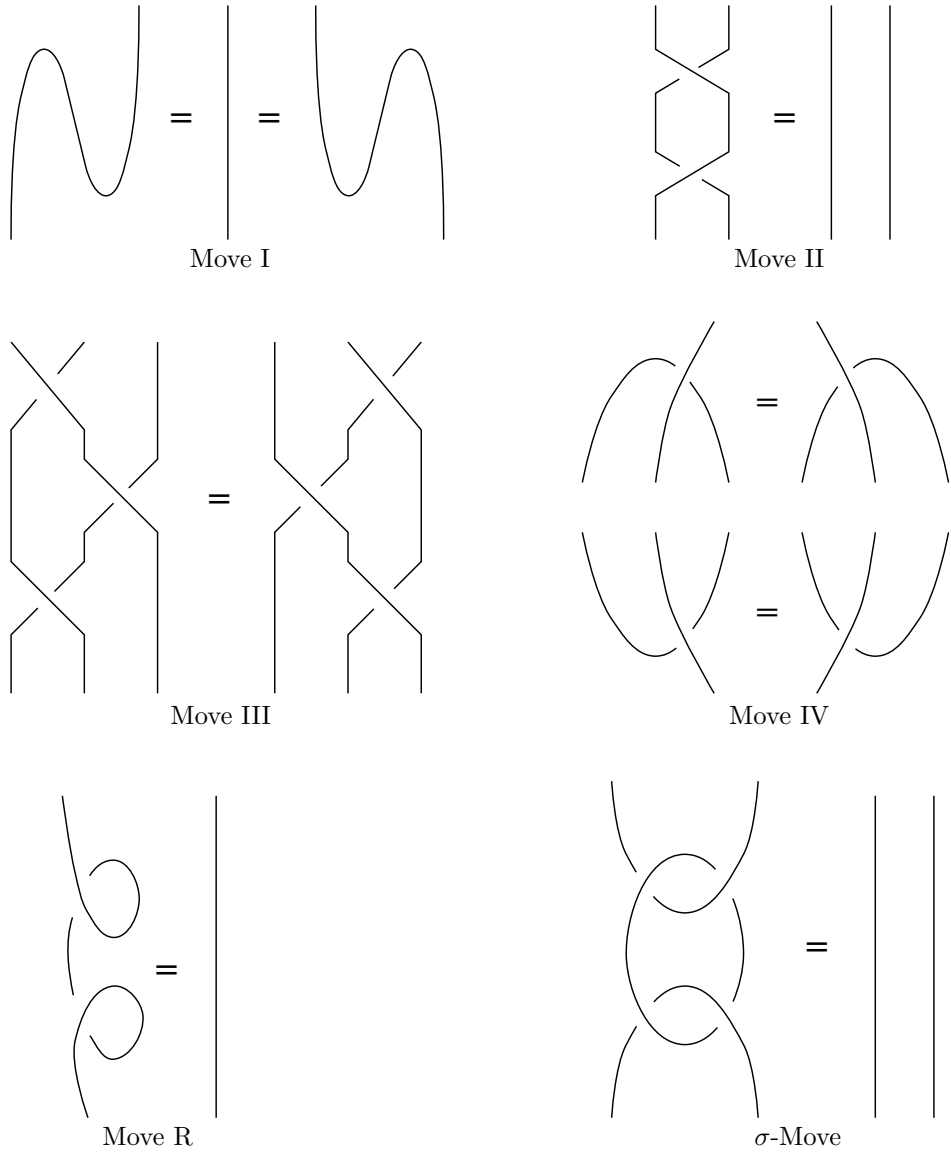}
	\caption{Equivalence of Tangles}
	\label{movesTgl}
\end{figure}

The functor $\kappa : ATGL \to Tgl$ in Figure \ref{theplan} is given by sending an even tangle $T \in ATGL$ to its equivalence class.  This functor is clearly surjective, so that given any admissible tangle $S \in Tgl$, we can find a preimage in $ATGL$.

So far, our categories are those used to define the Hennings TQFT on regular Hopf algebras.  Since quasi-Hopf algebras are not coassociative, to extend the Hennings TQFT to the case of quasi-Hopf algebras, we will apply brackets (or parenthesis) to our admissible tangles.  This produces a tangle category which is not a strict braided tensor category.  We want to preserve the equivalence of Theorem \ref{isomorphiccategorieslemma}, however, so we again restrict to a category of admissible tangles.

We begin by bracketing the integers.  By a bracketing, we mean adding parenthesis to the $n$ points so that everything is grouped in pairs.  We will use $\hat{n}$ to refer to an integer which has been bracketed in some fashion.  For example, one bracketing of $\hat{6}$ would be given by the following.
\begin{equation}
    \hat{6} = ((((1 2) (3 4)) 5) 6)
\end{equation}

    We first define a general category $\tanglep$ of bracketed tangles, where we apply brackets to elements of $TGL$.
\begin{definition}\label{tanglepdef}
    The category $\tanglep$ is defined as was $TGL$, with the following modifications.  Instead of integers, we use bracketed integers.  We also replace our diagrams by bracketed diagrams.  Finally, we add an additional generator, called $a$, which is depicted in Figure \ref{Tglagen}.
\end{definition}
The bracketed versions of the generators of $TGL$ are given in Figure \ref{TglPgens}.  We must also take care with the tensor product, as we must add parenthesis to the juxtaposition.  That is, $S \otimes T$ has parenthesis around all of $S$ and parenthesis around all of $T$.  Note also that Moves I, II, III, IV, and R will now have parenthesiss, but will be called by the same names as for $TGL$.

    \begin{figure}[ht]
       \centering
       \labellist
    	\pinlabel $\cap$ at 42 -10
    	\pinlabel $\cup$ at 155 -10 
    	\pinlabel $c$ at 255 -10
    	\pinlabel $c\inv$ at 345 -10 
    	\pinlabel $\id$ at 410 -10
    	\endlabellist
    	\includegraphics[scale=0.7]{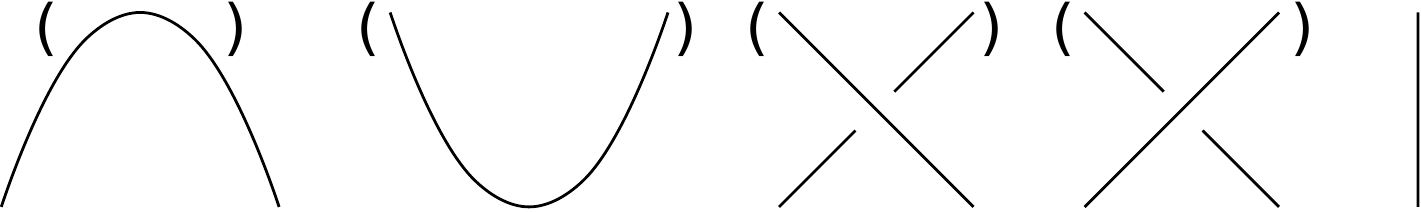}
    	\caption{Generators for the Bracketed Tangle Category}
    	\label{TglPgens}
    \end{figure}
    
    \begin{figure}[ht]
        \centering
    	\includegraphics[scale=0.7]{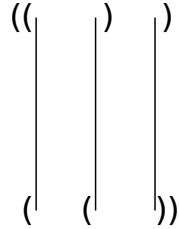}
    	\caption{The Generator $a$}
    	\label{Tglagen}
    \end{figure}
    
We can now define a category analagous to $ATGL$, called $A\tanglep$.  To define the even tangles of $A\tanglep$, we restrict to even integers as before, and bracket the source and targed objects in the following specific way.  In $A\tanglep$ where $n = 2m$ is an even number, we denote this special bracketing by $\widebar{n} = \widebar{2m}$. Labeling the $2m$ elements as $1^-, 1^+, ... , m^-, m^+$, we take
    \begin{equation}\label{specialbracket}
    	\widebar{n} = ( ...... ((1^-, 1^+), (2^-, 2^+)),  ... , (m^-, m^+)).
	\end{equation}
In $A\tanglep$, we still require end points $(j^-, j^+)$ to be either connected by a single component of the tangle, or via two components to the pair $(i^-, i^+)$ at the opposite end of the tangle.  There may also be closed components of the tangle, isomorphic to $S^1$.

The reader may verify that in $\tanglep$, all of the conditions for a braided tensor category are satisfied, including the Pentagon and Hexagon Axioms.  See \cite{Ka} for the relevant definitions.

Recall our diagram in Figure \ref{theplan}.  
\begin{lemma} \label{injectivetgls}
    We have an injective functor $b : ATGL \hookrightarrow \tanglep$.
\end{lemma}

\begin{proof}
	On objects, the functor is given by
	\begin{equation}
		n \mapsto \widebar{n},
	\end{equation}
	which is clearly injective.
	
	On morphisms, the functor is given by applying appropriate parenthesis to the generators, and then changing the parenthesis on the straight strands to match the bracketing from the left on $\widebar{2n}$.  This is well-defined by the Pentagon Axiom, and injective because two bracketed tangles with the same non-bracketed generators differ only by a change of parenthesis.  
\end{proof}

\begin{remark}
    The image of the composition $b \circ \kappa\inv : Tgl \to \tanglep$ in Figure \ref{theplan} is contained in the category $A\tanglep$.  To show that the TQFT is well-defined, we will need to show that the entire composition factors through the additional moves on $Tgl$.
\end{remark}

\subsection{The Decorated Tangle Category $\DTgl$}

The next step in understanding our diagram in Figure \ref{theplan} is to define a category of decorated, flat tangles, denoted $\DTgl$.  We use the superscript parenthesis to distinguish this category from one where the tangles are decorated by elements of a Hopf algebra.  We will also understand the functor $\Psi : \tanglep \to \DTgl$ in Chapter \ref{newhenningsch}.  By a ``flat'' tangle, we mean one from which over- and under-crossing information has been removed.  More formally, we have the following.

Let $H$ be a quasi-Hopf algebra.  The objects in $\DTgl$ are integers, as for $TGL$.  The morphisms in $\DTgl$ are given by $H$-labeled flat tangles.  Specifically, an $H$-labeled flat tangle is a pair $(D, a)$, where $D$ is a planar immersed curve in general position with $N$ ordered markings, and $a$ is an element of $H^{\otimes N}$.  We say that $D$ is in general position if there are no horizontal parts of the tangle, and if all extrema, markings, and crossings occur at different levels.  Furthermore, we picture this object as a formal sum of $\nu$ copies of the flat tangle $D$.  If 
\begin{equation}
    a = \sum a_1^\nu \otimes ... \otimes a_N^\nu,
\end{equation}
then the $N$ markings on the $j$-th copy of $D$ are labeled in order by  the factors of the $j$-th summand of $a$, $a_1^j \otimes ... \otimes a_N^j$.  

That is, we may consider a morphism in $DTgl$ to be a formal sum of a tangle composed of the generators in Figure \ref{DTglgens}.

\begin{figure}[ht]
    \centering
	\labellist
	\pinlabel $\cap$ at 42 -5
	\pinlabel $\cup$ at 155 -10
	\pinlabel $\times$ at 255 -5
	\pinlabel $g$ at 318 28
	\pinlabel ${g \in H}$ at 335 -5
	\endlabellist
	\includegraphics[scale=0.8]{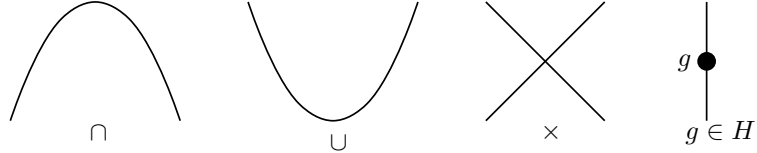}
	\caption{Generators for $\DTgl$}
	\label{DTglgens}
\end{figure}

\begin{definition} \label{decequivdef}
    Two decorated, flat tangles $T$ and $T'$ are equivalent if they are related by a sequence of the moves pictured in Figure \ref{movestrands2} and Figure \ref{movebeads}.
\end{definition}

\begin{figure}[ht]
    \centering
	\labellist
	\pinlabel {Move I} at 120 280
	\pinlabel {Move II} at 420 280
	\pinlabel {Move III} at 115 20
	\pinlabel {Move IV} at 420 20
	\endlabellist
	\includegraphics[scale=0.7]{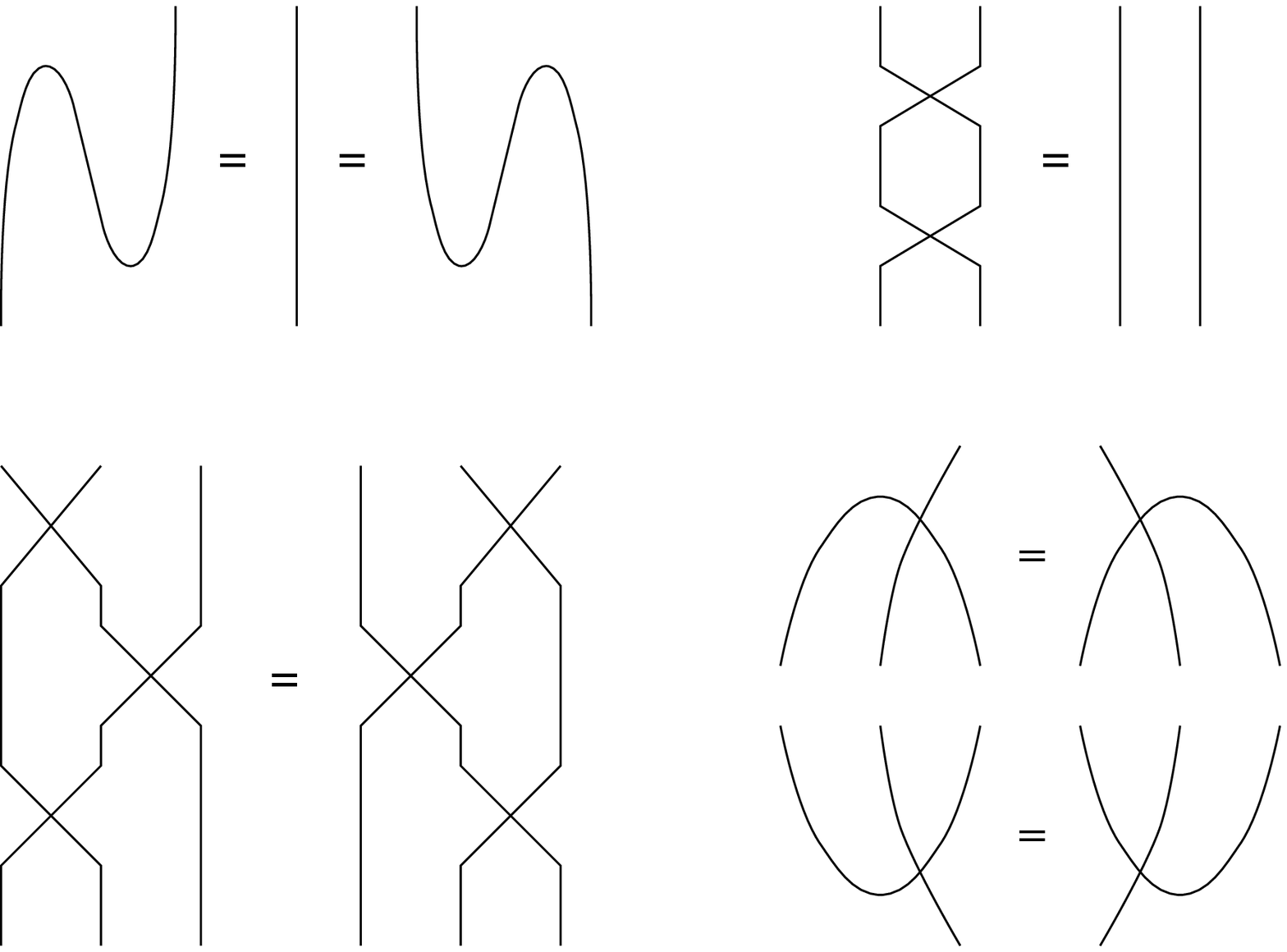}
	\caption{Equivalence of Flat Tangles}
	\label{movestrands2}
\end{figure}

\begin{figure}[ht]
    \labellist
	\pinlabel $x$ at 18 165
	\pinlabel $y$ at 18 140
	\pinlabel $yx$ at 92 155
	\pinlabel $x$ at 355 155
	\pinlabel $S(x)$ at 460 155
	\pinlabel $G$ at 250 155
	\pinlabel $x$ at 55 43
	\pinlabel $x$ at 170 15
	\pinlabel $x$ at 325 20
	\pinlabel $S(x)$ at 480 20
	\endlabellist
	\centering
	\includegraphics[scale=0.65]{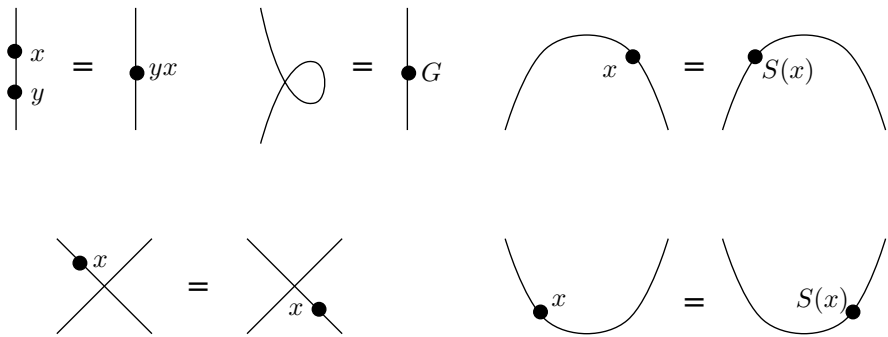}
	\caption{Equivalent Beaded Tangles in $DTgl$}
	\label{movebeads}
\end{figure}

As usual, composition is given by stacking, and the tensor product is given by juxtaposition.

\section{A Hennings TQFT Construction for Quasi-Hopf Algebras} \label{newhenningsch}

We now turn to the definition of the Hennings TQFT for quasi-Hopf algebras, and our aim is to extend the results in \cite{Ker1}.  We require the following assumptions.
\begin{definition}\label{assumptions}
    Let $H$ be a quasi-triangular ribbon quasi-Hopf algebra $H$.  Let $\lambda$ be a right cointegral on $H$.  We say that $H$ is a normalized unimodular quasi-Hopf algebra if the following additional assumptions hold.
    \begin{enumerate}
        \item $H$ is unimodular
        \item The elements $\alpha$, $\beta$ in $H$ are both invertible
        \item $\epsilon(\alpha) = \epsilon(\beta) = 1$ for the counit $\epsilon$
        \item $\lambda(\alpha v S\inv(\beta)) \lambda(\alpha v\inv S\inv(\beta)) = 1$
        \item $\lambda(\Lambda) = 1$ for an integral $\Lambda$
        \item The monodromy element ${\mathcal M}$ is nondegenerate in the sense of Definition \ref{Mnondegen}
    \end{enumerate}
\end{definition}
For the remainder of this paper, let $H$ be a normalized unimodular quasi-triangular quasi-Hopf algebra.  

The third and fourth assumptions in Definition \ref{assumptions} are possible by rescaling the right cointegral $\lambda$.  The other assumptions are discussed in Chapter \ref{prelim}.  These assumptions correspond directly to those used to define the Hennings algorithm in the ordinary Hopf case.  Furthermore, we can produce a quasi-Hopf algebra satisfying these conditions using the double construction for quasi-Hopf algebras.  See \cite{HN}.
    
    With Figure \ref{theplan} in mind, we define in Section \ref{quasi-Henningsdef} the functor $\Psi : \tanglep \to \DTgl$ and the functor $\hat{\Psi} : \DTgl \to Vect$.  In Section \ref{well-definedness} we discuss why $\Psi$ is well-defined, and then finally in Section \ref{diagramworks} we prove that the TQFT is well-defined; that is, we can factor through the additional equivalence relations in $Tgl$.

\subsection{Extending the Definition of the Hennings TQFT} \label{quasi-Henningsdef}

We now define the final two functors in the Hennings TQFT for quasi-Hopf algebras.  The TQFT $\tilde{\nu_H}$ is given by the overall composition in our diagram of Figure \ref{theplan}, and the only two functors we have not yet discussed are $\Psi : \tanglep \to \DTgl$ and $\hat{\Psi} : \DTgl \to Vect$.  

First, we define $\Psi : \tanglep \to \DTgl$.  Let $T : \widebar{m} \to \widebar{n}$ be a tangle diagram representing an equivalence class $[T] \in \tanglep$.  We map $[T]$ to a decorated flat tangle $[\hat{T}] \in \DTgl$ by systematically replacing the generators of $\tanglep$ in the diagram $T$ by diagrams in $\DTgl$.  Replacing the generators gives us a diagram $\hat{T}$, which in turn gives us an equivalence class $[\hat{T}] \in \DTgl$.   

First, we replace the crossing generators $c$ and $c\inv$ as in Figure \ref{crossingbeads}, where $\Rmat = \sum s_i \otimes t_i$.

\begin{figure}[ht]
	\centering
	\labellist
	\pinlabel $s_i$ at 185 40
	\pinlabel $t_i$ at 235 40
	\pinlabel {$\sum \limits_i$} at 155 30
	\pinlabel {$\sum \limits_i$} at 455 30
	\pinlabel $\bar{s_i}$ at 480 20
	\pinlabel $\bar{t_i}$ at 540 20
    \pinlabel $\Psi$ at 100 45
    \pinlabel $\Psi$ at 415 45
	\endlabellist
	\includegraphics[scale=0.65]{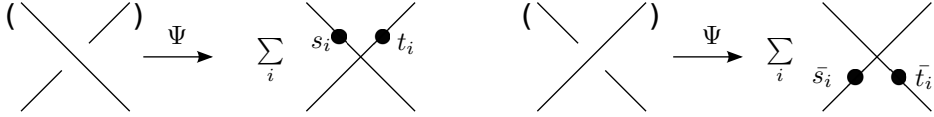}
	\caption{Replacement of Crossings}
	\label{crossingbeads}
\end{figure}

Next, we replace the generators $\cup$ and $\cap$ with beaded versions of these generators. See Figure \ref{maxmin}.  

\begin{figure}[ht]
	\centering
	\labellist
	\pinlabel $\beta$ at 145 50
	\pinlabel $\alpha$ at 500 5
    \pinlabel $\Psi$ at 110 45
    \pinlabel $\Psi$ at 400 45
	\endlabellist
	\includegraphics[scale=0.7]{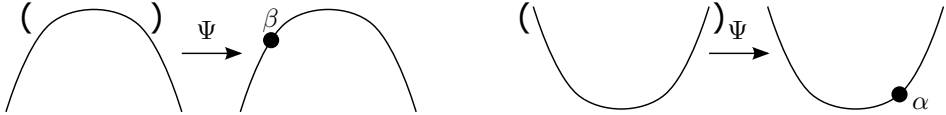}
	\caption{Replacement of Extrema}
	\label{maxmin}
\end{figure}

Finally, we replace the generator $a$ by appropriate factors of $\Phi$.   On the basic change of parenthesis on three strands, we apply the factors of $\Phi$ or $\Phi\inv$.  On a more complicated change of parenthesis, we apply the coproduct on the factors of $\Phi$ so that the bracketing of the strands in the diagram matches the bracketing given by the application of $\Delta$ to the factors of $\Phi$.  For instance, on the right-hand side of Figure \ref{parenchange}, we consider the first two strands to be the coproduct of a single strand, hence we must apply $(\Delta \otimes \id \otimes \id)(\Phi)$ to the four strands. 

\begin{figure}[ht]
	\centering
	\labellist
	\pinlabel $X_i$ at 114 25
	\pinlabel $Y_i$ at 147 25
	\pinlabel $Z_i$ at 180 25
	\pinlabel $X_i'$ at 409 25
	\pinlabel $X_i''$ at 441 25
	\pinlabel $Y_i$ at 475 25
	\pinlabel $Z_i$ at 509 25
    \pinlabel $\Psi$ at 95 50
    \pinlabel $\Psi$ at 390 50
	\endlabellist
	\includegraphics[scale=0.65]{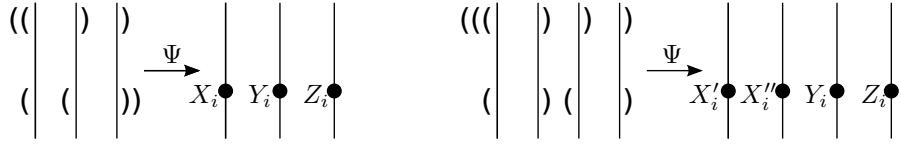}
	\caption{Replacement of Parenthesis}
	\label{parenchange}
\end{figure}

The results of applying $\Phi$ in Figure \ref{parenchange} should be given by a formal sum of tangles, but we have suppressed the summation symbol. This will be the case for the rest of our diagrams.

This replacement procedure produces a diagram $\hat{T}$ with $[\hat{T}] \in \DTgl$.  We then define $\Psi([T]) = [\hat{T}]$ on equivalence classes.

We next define the functor $\hat{\Psi} : \DTgl \to Vect$.  Suppose $[S] \in \DTgl$ is an equivalence class represented by a diagram $S: m \to n$.  Using the equivalence relations in the category $\DTgl$, presented in Figure \ref{movebeads}, we may find an equivalent tangle $S' \in \DTgl$ in which on each component, all of the beads have been collected into a single bead, and each component is arranged in one of the options depicted in Figure \ref{beadmap1}.  On the image $\Psi(b(ATGL))$, these are all the possibilities that can occur.

\begin{figure}[ht]
		\labellist
		\pinlabel $a_i$ at 98 53
		\pinlabel $(1)$ at 42 -10
		\pinlabel $b_j$ at 145 40
		\pinlabel $(2)$ at 187 -10
		\pinlabel $c_k$ at 375 72
		\pinlabel $(3)$ at 335 -10
		\pinlabel $p_n$ at 430 53
		\pinlabel $q_n$ at 517 53
		\pinlabel $(4)$ at 475 -10
		\endlabellist
		\centering
		\includegraphics[scale=0.6]{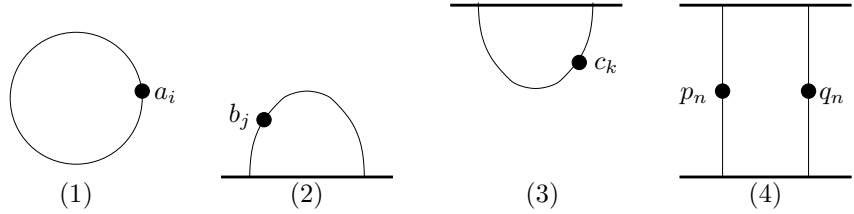}
		\caption{Hennings Bead Maps}
		\label{beadmap1}
\end{figure}

The map $f_{[S]} : H^{\otimes m/2} \to H^{\otimes n/2}$ corresponding to $[S] = [S']$ is given by the following.  For each pair of strands at either the top or bottom of the tangle, we have one factor of the underlying Hopf algebra $H$, and so we need one variable $x_i$.  

The maps corresponding to each of the types of tangles in Figure \ref{beadmap1} are given in the following list.

\begin{itemize}
	\item $(1) : {\bf k} \to {\bf k}$ sends $a_i \mapsto \lambda(a_i)$
	\item $(2) : {\bf k} \to H$ sends $1 \mapsto b_j$
	\item $(3) : H \to {\bf k}$ sends $x_k \mapsto \lambda(S(x_k)c_k)$
	\item $(4) : H \to H$ sends $x_n \mapsto p_n x_n S(q_n)$
\end{itemize}

Now $f_{[S]}$ is the tensor product of the relevant maps.  The functor $\hat{\Psi}$ sends objects $2m$ to $H^{\otimes m}$ and morphisms $\hat{\Psi}([S]) = f_{[S]}$.  Note that on the image $b(ATGL) \subseteq \tanglep$, all objects are even integers and can be written as $n = 2m$.

\begin{lemma}\label{hatpsiwd}
	The functor $\hat{\Psi}$ is well-defined on the image $\Psi(b(ATGL))$.
\end{lemma}
\begin{proof}
    The idea of this proof is the same as for ordinary Hopf algebras, and the proof in that case is found in \cite{Ker1, CK, KR}.  Since the equivalent tangle $S'$ in standard position is unique up to $S^2$ applied to elements of the bead on the circular component, we need only check that
    \begin{align}
		\lambda(S^2(b)a) = \lambda(ab).
	\end{align}
	This is one of the assumptions required for the cointegral $\lambda$.  See \eqref{cointegralcommute}.
\end{proof}

\begin{remark}
    We recover the original Hennings link invariant \cite{Hennings} by using an ordinary Hopf algebra $H$ and applying a suitable normalization.
\end{remark}

As an introductory example, we present the following, which will be used to simplify our calculations in the following sections.
\begin{lemma} \label{singlestrandv}
    A single-strand twist maps to an untangled strand decorated with the special element $v$ under the functor $ \Psi : \tanglep \to \DTgl$.
\end{lemma}

\begin{proof}
    A single strand twist is decorated with beads as in Figure \ref{newu}.  Using the rules for collecting beads, we see that this decorated tangle is equivalent to a tangle with a single bead labeled by the special element $u$.
    \begin{figure}[ht]
        \centering
        \labellist
    	\pinlabel $\bar{X_i}$ at -10 150
    	\pinlabel $\bar{Y_i}$ at 59 150
    	\pinlabel $\bar{Z_i}$ at 128 150
    	\pinlabel $\beta$ at 75 175
    	\pinlabel $\alpha$ at 65 25
    	\pinlabel $s_j$ at 10 120
    	\pinlabel $t_j$ at 65 120
        \pinlabel $\Psi$ at 210 115
    	\endlabellist
    	\includegraphics[scale=0.6]{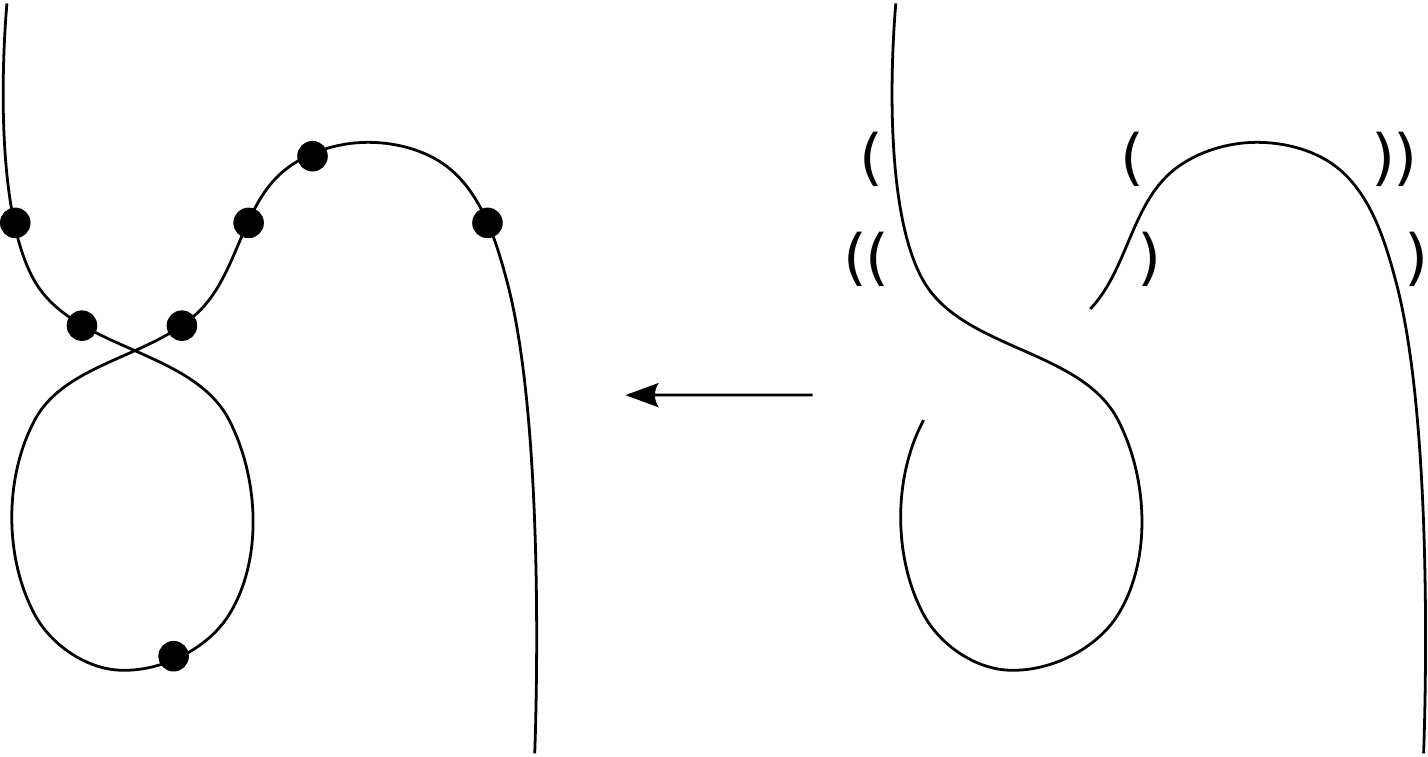}
    	\caption{Diagram Correpsonding to $u$}
    	\label{newu}
    \end{figure}
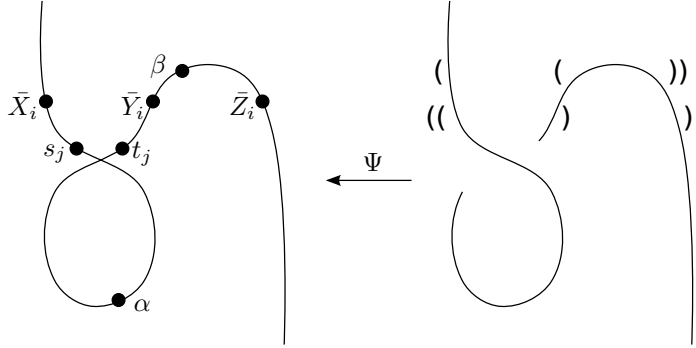
    
    We then remove the twist from the decorated, flat tangle by applying the special element $G\inv$.  We have that $G\inv u = v$, which completes the proof.  See Figure \ref{velt1}.
    \begin{figure}[ht]
        \centering
    	\labellist
    	\pinlabel $u$ at 60 100
    	\pinlabel $u$ at 120 90
    	\pinlabel $G\inv$ at 132 45
    	\pinlabel ${G\inv u = v}$ at 252 65
    	\endlabellist
    	\includegraphics[scale=0.6]{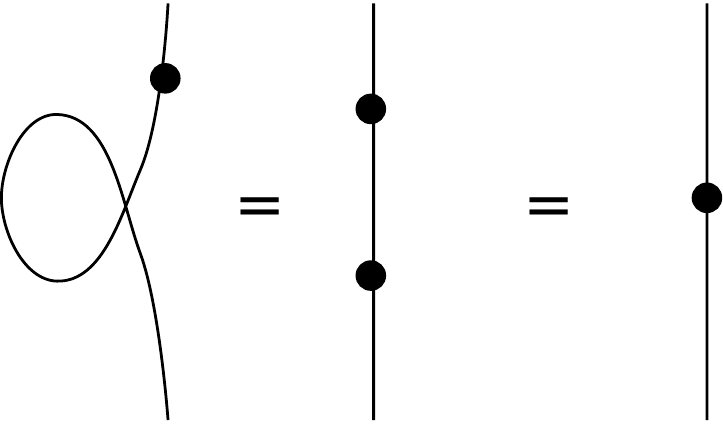}
    	\caption{Gathering the Beads}
    	\label{velt1}
    \end{figure}
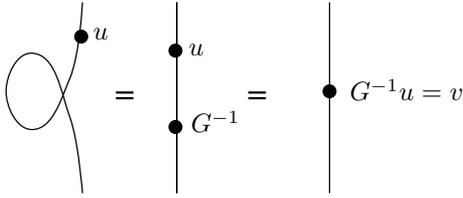
\end{proof}
Similarly, one can show that a negative twist maps under $\Psi$ to a straight strand labeled with $v\inv \in H$.

\subsection{Proof that $\Psi$ is Well-Defined}\label{well-definedness}

We have defined the functors $\Psi$ and $\hat{\Psi}$ and proven that $\hat{\Psi}$ is well-defined.  We now prove that $\Psi$ is also well-defined.  Recall that $H$ is a normalized unimodular quasi-Hopf algebra, in the sense of Definition \ref{assumptions}.  Recall also from Definition \ref{tanglepdef} that the equivalence of tangles in $\tanglep$ is generated by the following.
    	\begin{enumerate}
            \item Rebracketing, or the Pentagon and Hexagon Axioms
			\item Bracketed Moves I - IV
            \item Bracketed Move R
        \end{enumerate}
Lemma \ref{hatpsiwd} gives us that $\hat{\Psi}$ is well-defined, so it remains to show that $\Psi$ is well-defined.  That is, if $T$ and $T'$ are equivalent diagrams under the above moves, $\Psi(T)$ and $\Psi(T')$ are equivalent diagrams in $\DTgl$. 

\begin{theorem}\label{step1psi}
    The functor $\Psi$ is well-defined, hence the composition $\Xi = \hat{\Psi} \circ \Psi \circ b$ is also well-defined.
\end{theorem}

\begin{proof}
    We will prove a sequence of lemmas that will show $\Psi$ factors through each of the moves above.  In Lemma \ref{moves=andIIlemma}, we will show that $\Psi$ factors through Moves I and II.  In Lemma \ref{theIIImovelemma}, we will show that $\Psi$ factors through Move III.  In Lemma \ref{rinvnewtangle}, we prove that $\Psi$ factors through a move equivalent to Move IV.  Finally, in Lemma \ref{moveRlemma}, we prove that $\Psi$ factors through bracketed Move R.  
    
        If two tangles are isotopic, they are decorated with the same beads under $\Psi$.  The only cause for concern is moving a maximum or minimum through a change of parenthesis, but this is possible with the assumptions on $\tanglep$ guaranteeing that $\tanglep$ is a braided tensor category. 
	
    Internal rebracketing of tangles is equivalent in the tangle category to using the Pentagon Axiom \eqref{pentagonaxiom}.  Hence the functor $\Psi$ factors through any internal rebracketing.

\end{proof}

We begin with Moves I - IV.

\begin{lemma} \label{moves=andIIlemma}
    The functor $\Psi$ factors through Moves I and II, depicted in Figure \ref{firstmoves}.
	\begin{figure}[ht]
		\begin{subfigure}{0.45\textwidth}
			\centering
			\includegraphics[scale=0.5]{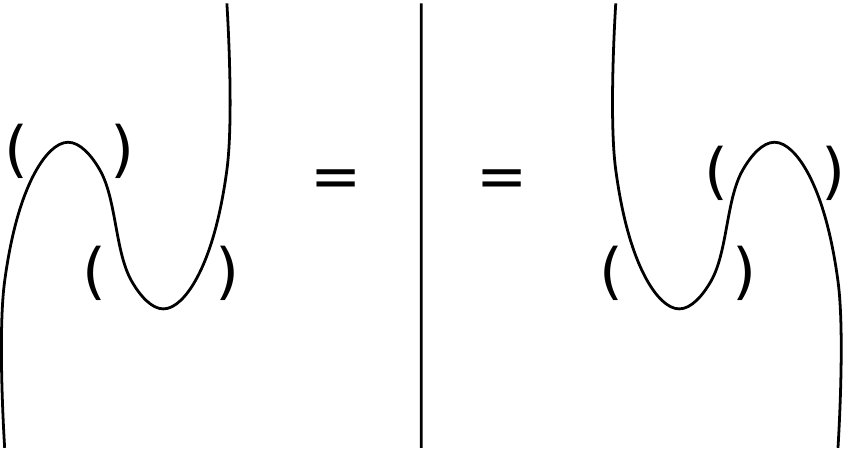}
			\caption{Move I}
			\label{the=move}
		\end{subfigure}
		\begin{subfigure}{0.45\textwidth}
			\centering
			\includegraphics[scale=0.5]{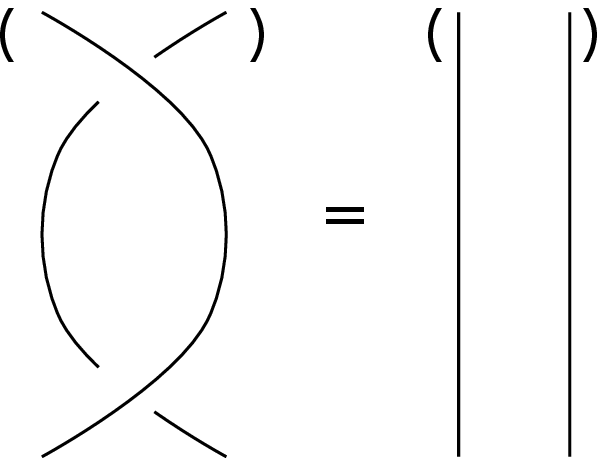}
			\caption{Move II}
			\label{IImove}
		\end{subfigure}
		\caption{Moves I and II}
		\label{firstmoves}
	\end{figure}
\end{lemma}

\begin{proof}
	We begin with Move I,  also called a cancellation.  See Figure \ref{the=beads}.  The beaded diagram for the left-hand side of Figure \ref{the=move} corresponds to the equation
	\begin{equation}\label{first=eqn}
		\sum X_i \beta S(Y_i) \alpha Z_i = 1
	\end{equation}
	when collected at the bottom of the diagram.  In $\DTgl$, we use Move I to cancel a maximum followed directly by a minimum, or a ``wiggle''.  This cancellation gives a straight strand with a single bead, labeled with the identity, which is equivalent to an unlabeled straight strand.   The right-hand side is similar.  Thus, the functor factors through Move I.  

	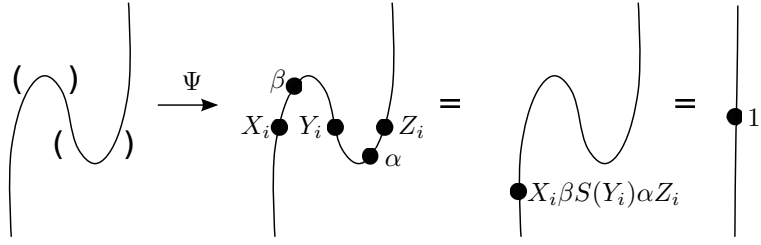
\begin{figure}[ht]
		\centering
		\labellist
		\pinlabel $\beta$ at 147 85
		\pinlabel $\alpha$ at 210 42
		\pinlabel $X_i$ at 135 60
		\pinlabel $Y_i$ at 165 60
		\pinlabel $Z_i$ at 220 60
		\pinlabel {$X_i \beta S(Y_i) \alpha Z_i$} at 325 23
        \pinlabel $\Psi$ at 100 85
		\pinlabel 1 at 407 65
		\endlabellist
		\includegraphics[scale=0.7]{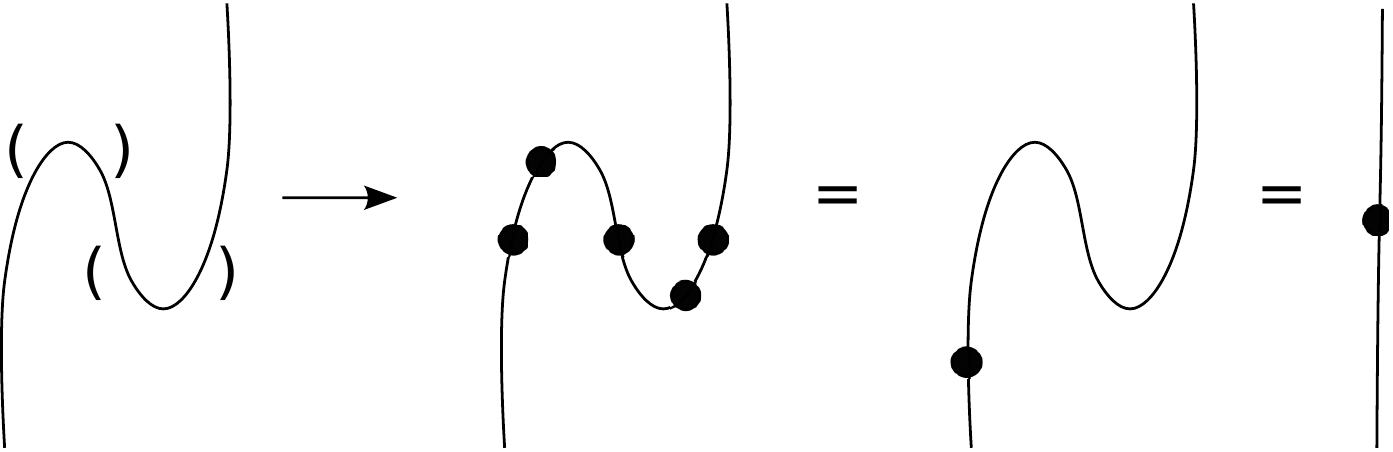}
		\caption{The Functor $\Psi$ Applied to Move I}
		\label{the=beads}
	\end{figure}
	
	For Move II, no change in parenthesis is needed, hence the proof is identical to the ordinary Hopf case.  The functor $\Psi$ applies the beads for $\Rmat$ and $\Rmat\inv$, which cancel.  Thus, the functor factors through Move II.
\end{proof}

\begin{lemma}\label{theIIImovelemma}
    The functor $\Psi$ factors through Move III, depicted in Figure \ref{IIImove}.
	\begin{figure}[ht]
		\centering
		\includegraphics[scale=0.6]{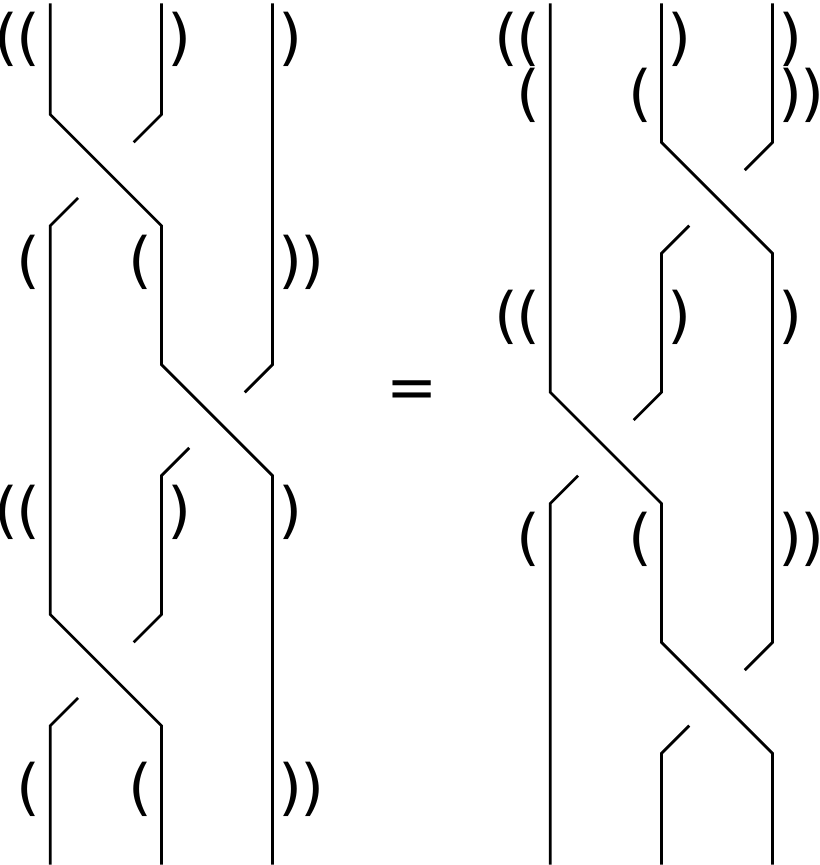}
		\caption{Move III}
		\label{IIImove}
	\end{figure}
\end{lemma}

\begin{proof}
	Note that we must ensure that the parenthesis at the top and bottom of each diagram are the same, so that the interior of one side of the diagram may be replaced by the other.  This move corresponds to the quasi-Yang Baxter equation,
	\begin{equation} \label{IIIResult}
		\Phi_{321}\Rmat_{23}\Phi\inv_{231}\Rmat_{13}\Phi_{213}\Rmat_{12} = \Rmat_{12}\Phi_{312}\Rmat_{13}\Phi\inv_{132}\Rmat_{23}\Phi_{123}.
	\end{equation}
Applying the functor $\Psi$ on both sides of the equality in Figure \ref{IIImove} and collecting the resulting beads on each strand gives us precisely \eqref{IIIResult}.  Hence $\Psi$ factors through Move III.
\end{proof}

In order to prove equivalence under Move IV, we prove an equivalent lemma giving us an alternate definition for the negative crossing generator $c\inv$ of $\tanglep$.

\begin{figure}[ht]
    \centering
	\includegraphics[scale=0.7]{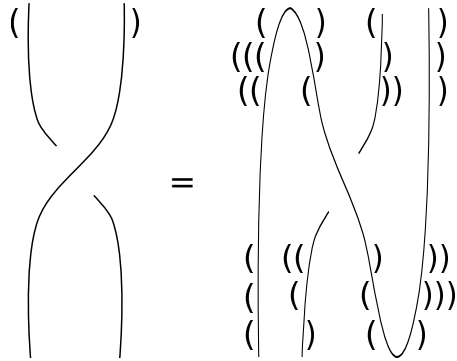}
	\caption{The Definition of $d$}
	\label{Rinv}
\end{figure}

\begin{lemma} \label{rinvnewtangle}
	The tangle in Figure \ref{Rinv}, which we temporarily label $d$, may be substituted for the negative crossing $c\inv$ in $\tanglep$.  
\end{lemma}

\begin{proof}
	We first draw in Figure \ref{rrinv1} the tangle corresponding to $cd$.
	\begin{figure}[ht]
		\centering
		\includegraphics[scale=0.5]{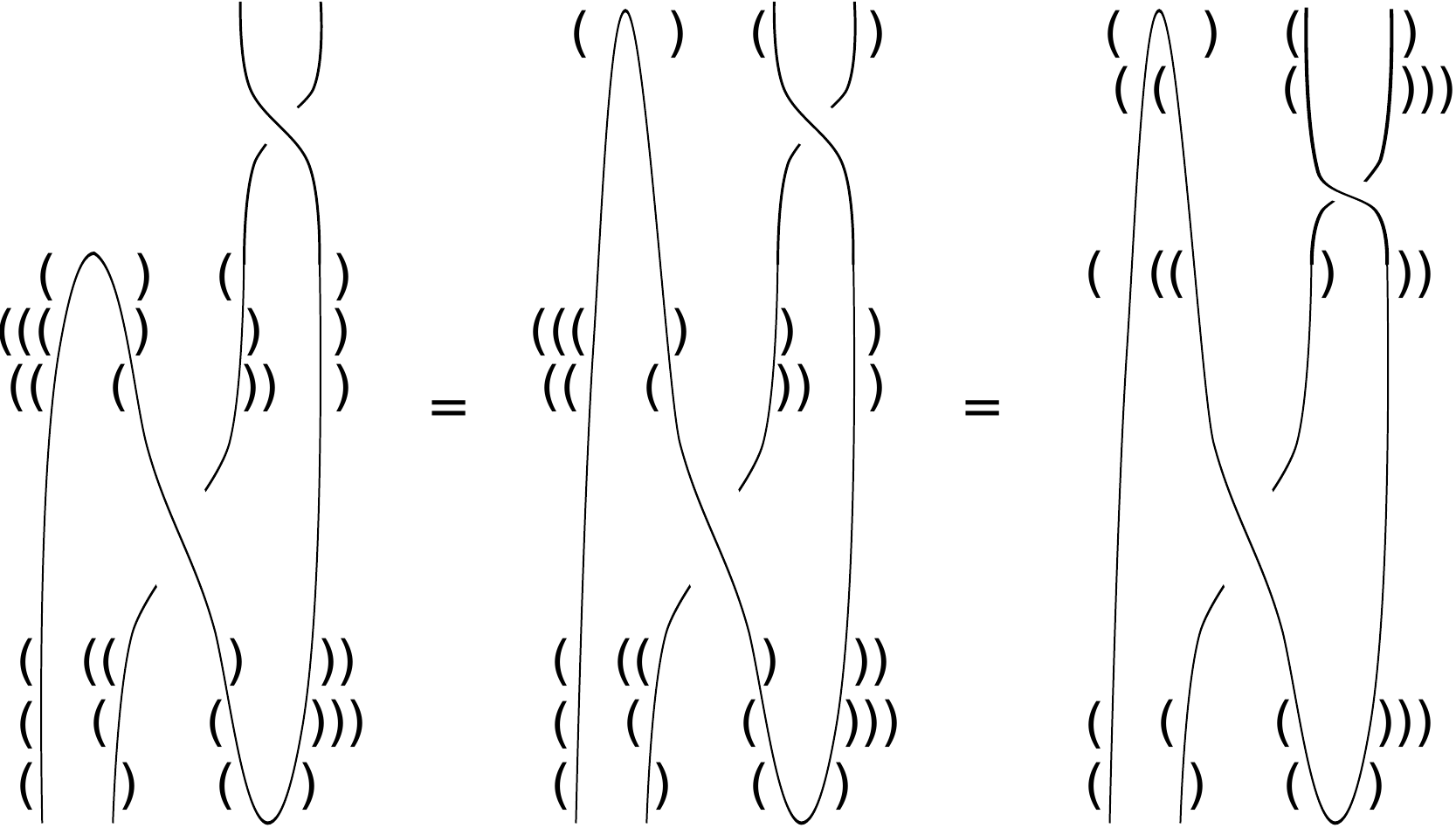}
		\caption{Steps to Prove $cd$ is Trivial: Part 1}
		\label{rrinv1}
	\end{figure}
	
	\begin{figure}[ht]
		\centering
		\includegraphics[scale=0.5]{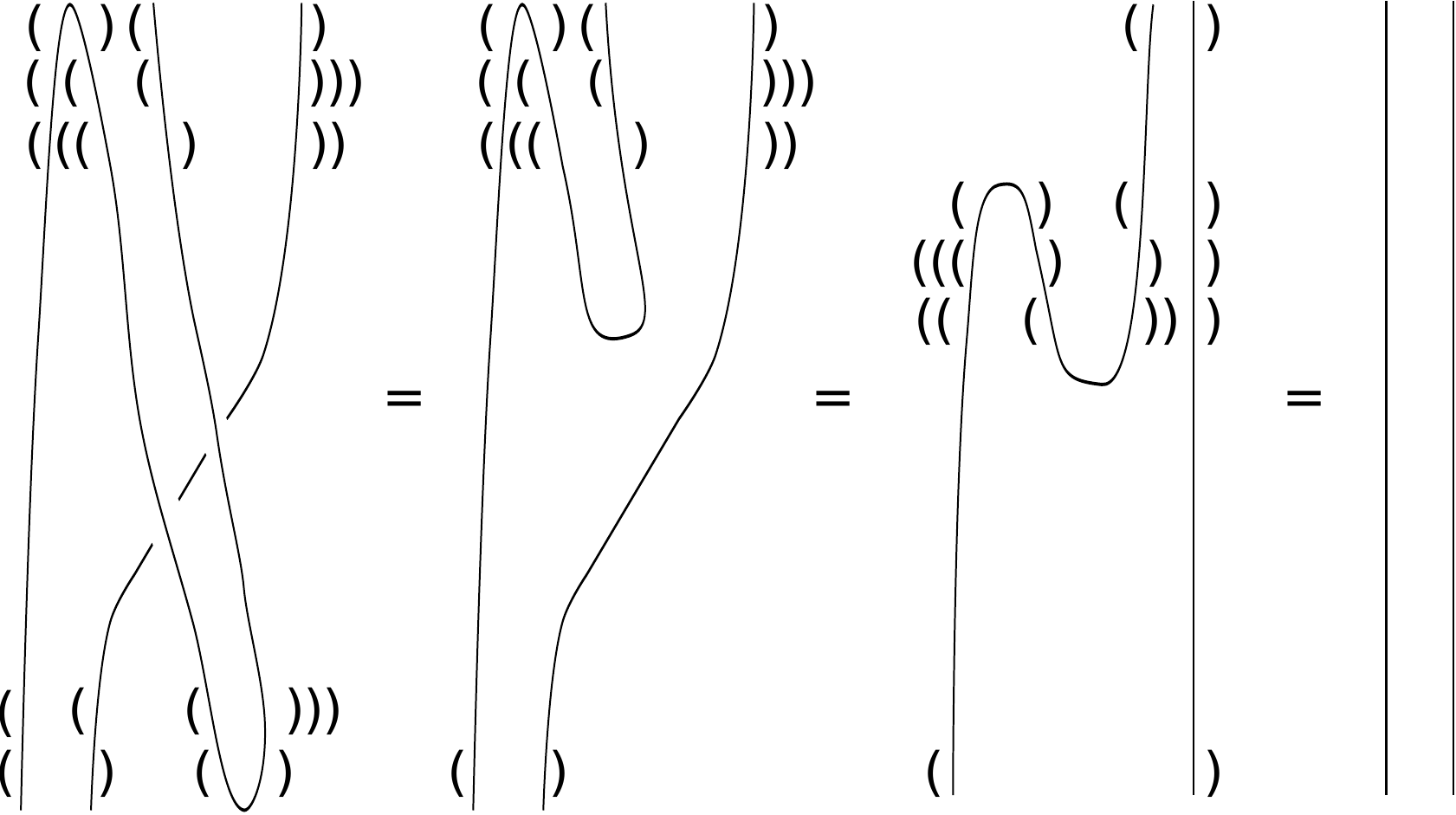}
		\caption{Steps to Prove $cd$ is Trivial: Part 2}
		\label{rrinv2}
	\end{figure}
	We elaborate on Figure \ref{rrinv1}.  For the first equality in the figure, we have used an isotopy to stretch the maximum upwards to the top of the diagram.  In the second equality, we have changed the interior parenthesis.  Our goal in this second step is to produce appropriate parenthesis to apply the Hexagon axiom in the tangle category.  Next consider Figure \ref{rrinv2}.

	We moved from Figure \ref{rrinv1} to Figure \ref{rrinv2} by applying the Hexagon Axiom. The first equality is an isotopy.  The second equality is a change of parenthesis.  In the third equality, we use Move I on the first strand.
	
	The argument to show that $d c$ is also trivial is essentially the same.
	
\end{proof}

\begin{remark}
	We will use the notation $c\inv$ for both the original negative crossing generator as well as for the tangle on the right-hand side of Figure \ref{Rinv}.
\end{remark}

As we have already discussed the validity of Move II for an abstract definition of $\Rmat\inv = \sum \bar{s_i} \otimes \bar{t_i}$, we may now use the beads produced by applying $\Psi$ to the tangle in Figure \ref{Rinv} to produce an explicit formula for $\Rmat\inv$.  We have the following.
\begin{align} 
	\Rmat \inv &=  \sum X_k X_i \beta S(Y_i) S(s_j) S(Y_k'')S(Y_m) \alpha Z_m Z_k \otimes X_m Y_k' t_j Z_i \label{rinvalg}\\
	& = \sum X_k X_i \beta S(Y_m Y_k'' s_j Y_i) \alpha Z_m Z_k \otimes X_m Y_k' t_j Z_i \nonumber
\end{align}

Using only moves we have already verified, Move IV is equivalent to Lemma \ref{rinvnewtangle}.  Hence, our the functor $\Psi$ factors through Move IV.

\begin{lemma}\label{moveRlemma}
    The functor $\Psi$ factors through the framed First Reidemeister Move, or Move R.
\end{lemma}

\begin{proof}
        Consider Figure \ref{1stRMove}.  Recall (or see Lemma \ref{singlestrandv}) that all necessary changes of parenthesis are included in the labeling by either $v$ or $v\inv$ for the twist or its inverse.  The functor $\Psi$ maps this tangle to the identity map, hence the functor factors through this move as desired.
	
    \begin{figure}[ht]
		\centering
		\labellist
        \pinlabel $\Psi$ at 135 135 
		\pinlabel $v\inv$ at 185 165
		\pinlabel $v$ at 190 70
		\pinlabel ${vv\inv=1}$ at 340 120
		\endlabellist
		\includegraphics[scale=0.5]{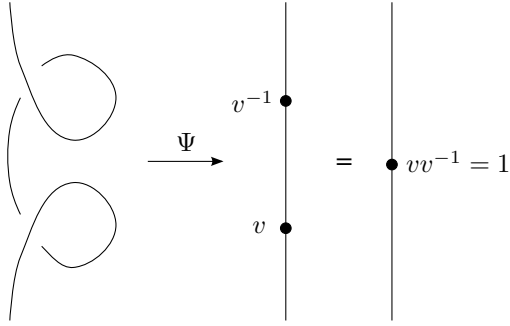}
		\caption{Framed 1st Reidemeister Move}
		\label{1stRMove}
	\end{figure}

\end{proof}

This concludes our sequence of lemmas.

\subsection{Proof of the Main Theorem: The TQFT $\tilde{\nu_H}$ is Well-Defined}\label{diagramworks}

Finally, we turn to our proof of the main theorem mentioned in the Introduction.  The only remaining ambiguity in the diagram of Figure \ref{theplan} is the inverse image of the map $\kappa : ATGL \to Tgl$.  We will demonstrate that given two equivalent tangles in $Tgl$, their images in $Vect$ are identical.  That is, we would like to see that the following diagram commutes.
    $$\xymatrix{ ATGL \ar@{>>}[dr]_\kappa \ar[rr]^\Xi & & Vect \\
        & Tgl \ar[ur] & }$$
The additional relations in $Tgl$ are the following.  See Section \ref{tangleCatDefs}.
        \begin{enumerate}
    		\item {The Fenn-Rourke Move}
			\item {The Modified First Kirby Move}
			\item The $\sigma$-move
		\end{enumerate}
        
\begin{theorem} \label{thisisthemainidea}
    Given a normalized unimodular quasi-Hopf algebra $H$ in the sense of Definition \ref{assumptions}, the TQFT $\tilde{\nu}_H$ is well-defined.
\end{theorem}

\begin{proof}
    Theorem \ref{step1psi} gives that $\Xi$ is well-defined.  As we just discussed, it remains to show that the composition with the inverse image of $\kappa$ is also well-defined.  Again, we will use a series of lemmas.
		
	We first prove that the Modified First Kirby Move is satisfied; Lemma \ref{RemoveHopfLink} proves that our composition $\Xi$ factors through this move.  Lemma \ref{n-strand-FR} proves that $\Xi$ factors through the Fenn-Rourke Move as well.  Finally,  Lemma \ref{sigmamovelemma} gives that our functor factors through the $\sigma$-move.  These are all of the additional moves, hence this will conclude the proof.
	
\end{proof}

We will consider the moves in the given order.

\begin{lemma} \label{RemoveHopfLink}
    The functor $\Xi$ factors through the addition or removal of a Hopf link where one component has zero framing, or the Modified 1st Kirby Move.
\end{lemma}

\begin{proof}

    Via a result in \cite{CK}, it suffices to show that two disjoint closed circles, one with $+1$ framing and the other with $-1$ framing, cancel.  We have the tangles in Figure \ref{1stKirby2}, with the first part of our functor applied.
    
	\begin{figure}[ht]
		\centering
		\labellist
		\pinlabel $+1$ at 75 50
		\pinlabel $-1$ at 125 50
		\pinlabel $\beta$ at 317 90
		\pinlabel $\beta$ at 425 90
		\pinlabel $v$ at 380 45
		\pinlabel $v\inv$ at 485 45
		\pinlabel $\alpha$ at 495 5
		\pinlabel $\alpha$ at 383 5
		\pinlabel $\Psi$ at 260 60
		\endlabellist
		\includegraphics[scale=0.6]{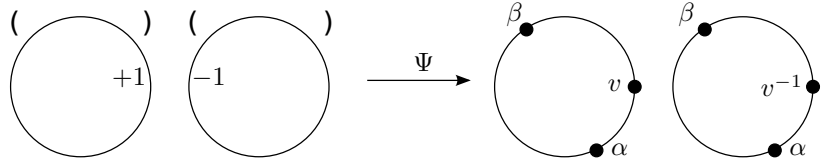}
		\caption{A Tangle for the Modified 1st Kirby Move}
		\label{1stKirby2}
	\end{figure}
	
	Collecting the beads on the right-hand side of each closed component, we see that the resulting field element is
	\begin{align*}
		\lambda(\alpha v S\inv(\beta)) \lambda( \alpha v\inv S\inv(\beta)),
	\end{align*}
	which is trivial by assumption.
\end{proof}

Next, to prove that $\Xi$ factors through the Modified Second Kirby Move, we prove instead that the Fenn-Rourke Move \cite{FR}, modified for tangles instead of links, holds.  This will be done in two steps.  We first prove that the Fenn-Rourke Move holds for a single vertical strand.

\begin{lemma} \label{one-strand-FR}
    The relation depicted in Figure \ref{frsetup} holds under the functor $\Xi$, where the labeling ``$-1$'' on the circle represents a negative twist, or a framing number of $-1$.
	\begin{figure}[ht]
		\centering
		\labellist
		\pinlabel $-1$ at 60 135
		\pinlabel $-1$ at 125 75
		\pinlabel $-1$ at 480 100
		\endlabellist
		\includegraphics[scale=0.6]{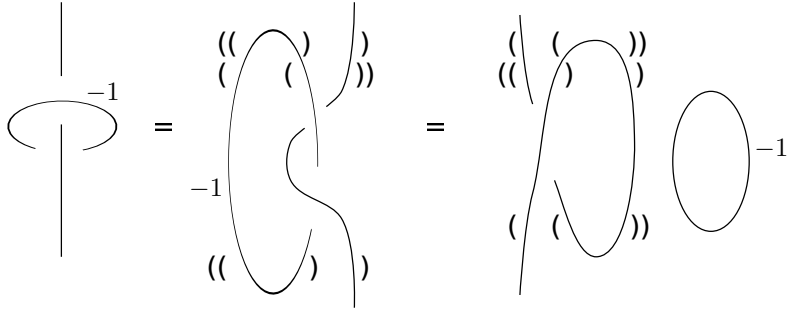}
		\caption{Equivalent Tangles}
		\label{frsetup}
	\end{figure}
\end{lemma} 

\begin{proof}
	For the first equality in Figure \ref{frsetup} we use Move IV, and therefore the functor $\Xi$ satisfies this first relation.  We must investigate the second equality.  The tangle on the right-hand side of Figure \ref{frsetup} maps under $\Psi$ to a single strand with one bead labeled $v$, and one closed circle with one bead on the right-hand side labeled $v\inv$, along with the usual beads at the minimum and maximum.  That is, applying $\hat{\Psi}$ we have $\lambda(\alpha v\inv S\inv(\beta)) v$.  We wish to show that the middle diagram of Figure \ref{frsetup} maps to this same element under $\hat{\Psi} \circ \Psi$.  Consider Figure \ref{frsetupbeads}.  Recall that $v\inv$ is central, with $S(v\inv) = v\inv$, hence this bead may be moved freely through any decorated tangle.  Also recall that we suppress summation indices for $p_L$ and $q_L$ and write $p_L = \sum \tilde{p}^1 \otimes \tilde{p}^2.$
	
	\begin{figure}[ht]
		\centering
		\labellist
        \pinlabel $\Psi$ at 150 110
		\pinlabel $\tilde{p}^1$ at 240 150
		\pinlabel $\tilde{p}^2$ at 295 150
		\pinlabel $s_i$ at 242 130
		\pinlabel $t_i$ at 280 130
		\pinlabel $v\inv$ at 280 95
		\pinlabel $s_j$ at 232 72
		\pinlabel $t_j$ at 272 72
		\pinlabel $\tilde{q}^1$ at 242 45
		\pinlabel $\tilde{q}^2$ at 286 45
		\endlabellist
		\includegraphics[scale=0.7]{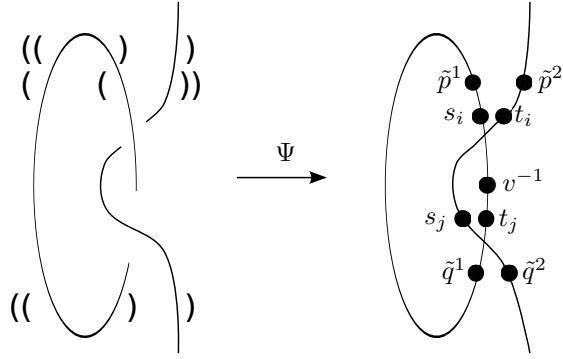}
		\caption{Applying the Map $\Psi$}
		\label{frsetupbeads}
	\end{figure}
	
	We collect these beads as follows.  On the closed circle, we have one bead labeled $\tilde{q}^1 t_j v\inv s_i \tilde{p}^1$ and on the straight strand we have one bead labeled $\tilde{q}^2 s_j t_i \tilde{p}^2$.  Applying the map $\hat{\Psi}: \DTgl \to H$-mod, we have
	\begin{equation}
		\sum \lambda(\tilde{q}^1 t_j v\inv s_i \tilde{p}^1) \tilde{q}^2 s_j t_i \tilde{p}^2. \label{negtwist1}
	\end{equation}
	We know that $v$ is in the center of $H$, so we may move it anywhere in the product.  Also, we substitute the element ${\cal M} = \Rmat^{\rm op} \Rmat = \sum t_j s_i \otimes s_j t_i$ and use the notation  ${\cal M} = \sum {\cal M}^1 \otimes {\cal M}^2$.  We rewrite \eqref{negtwist1} as
	\begin{equation}
		\sum \lambda(v\inv \tilde{q}^1 {\cal M}^1 \tilde{p}^1) \tilde{q}^2 {\cal M}^2 \tilde{p}^2. \label{negtwist2}
	\end{equation}

We simplify \eqref{negtwist2} into the desired result.
	\begin{align}
		& \sum  \lambda(v\inv \tilde{q}^1 {\cal M}^1 \tilde{p}^1) \tilde{q}^2 {\cal M}^2 \tilde{p}^2 \nonumber\\
		 &= \sum  \lambda(v\inv \tilde{q}^1 {\cal M}^1 \tilde{p}^1)v v\inv \tilde{q}^2 {\cal M}^2 \tilde{p}^2 \nonumber \\
		 & = v  (\lambda \otimes \id)[(v\inv \otimes v\inv) q_L {\cal M} p_L)]  \label{frsetupeq1}
	\end{align}
Next, $\mathcal{M}$ nondegenerate implies that ${\cal M} = \Delta (v\inv) (v \otimes v)$ (see \cite{CK}), so we may rewrite \eqref{frsetupeq1} as follows.
	\begin{align}
		&= v (\lambda \otimes \id)[(v\inv \otimes v\inv) q_L \Delta(v\inv)(v \otimes v) p_L] \nonumber\\
		&= v (\lambda \otimes \id)(q_L \Delta(v\inv) p_L) \label{frsetupeq3}
	\end{align}
Finally, the property of $\lambda$ in Lemma \ref{cointHelper} allows us to rewrite \eqref{frsetupeq3}.
	\begin{equation*}
		= v \cdot \lambda(\alpha v\inv S\inv(\beta))
	\end{equation*}
This is the desired result.
	
\end{proof}

We extend this result to $n$ strands by taking an $n$-fold coproduct in $\tanglep$ on the straight strand.  We can define an $n$-fold coproduct on a strand in a diagram of a tangle in $\tanglep$ by designating a particular bracketing for the resulting object.  We will designate the bracketing from the right and define the coproduct inductively.
\begin{definition}
    The $n$-fold coproduct on a strand $S$ in a tangle $T$ is denoted $\deltaR{n}(S)$ and is defined inductively as follows.
    \begin{align*}
        &\deltaR{1}(S) = \Delta(S) = (S \otimes S) \\
        &\deltaR{2}(S) = (1 \otimes \Delta)(\Delta(S)) \\
        & \deltaR{n}(S) = (1 \otimes ... \otimes 1 \otimes \Delta)(\deltaR{n-1}(S))
    \end{align*}
\end{definition}
Strands which have been bracketed from the right in this fashion will be denoted by a pair of strands whose bracketing has a subscript $R$.  We will use $(n)$ over the pair of strands to indicate that there are $n$ strands.

We may also define a coproduct on strands in $\DTgl$ by doubling unlabeled strands and by using the coproduct in our quasi-Hopf algebra on beaded strands.  See Figure \ref{DTglCoprod}.  We then define an $n$-fold coproduct in the same manner as for $\tanglep$.
\begin{figure}[ht]
    \centering
	\labellist
	\pinlabel $\Delta$ at 125 365
	\pinlabel $\Delta$ at 140 70
	\pinlabel $\Delta$ at 140 205
	\pinlabel $h$ at -10 355
	\pinlabel $h'$ at 230 355
	\pinlabel $h''$ at 285 355
	\pinlabel {$\sum \limits_{(h)}$} at 200 345
	\pinlabel $(f\inv)^1$ at 205 65
	\pinlabel $(f\inv)^2$ at 260 40
	\pinlabel $f^1$ at 270 200
	\pinlabel $f^2$ at 310 200
	\endlabellist
	\includegraphics[scale=0.7]{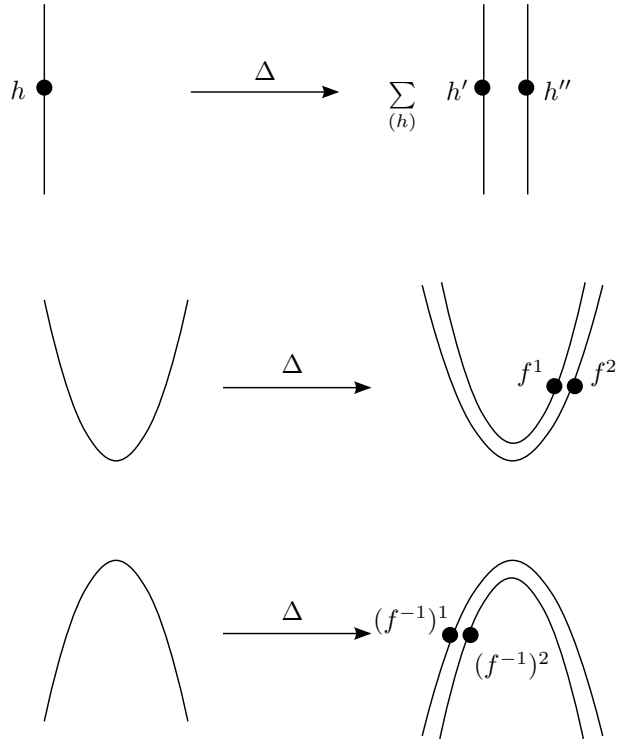}
	\caption{A Coproduct on Strands in $DTgl$}
	\label{DTglCoprod}
\end{figure}  
We can show that this notion of coproduct commutes with our functor $\Psi : \tanglep \to \DTgl$.  Hence, to obtain the Fenn-Rourke Move on $n$ strands, we take the $n$-fold coproduct of the vertical strand.  We consider a more general case first.

\begin{lemma} \label{n-strand-FR-helper}
    Let $a \in H$ be central with the property that $S(a) = a$.  Let $\Omega = \Omega^1 \otimes \Omega^2$ be defined as follows.
	\begin{align}
		\Omega = q_L {\cal M} p_L \label{omegadef}
	\end{align}
	Under the functor $\hat{\Psi} \circ \Psi : \tanglep \to Vect$, the strands of the tangle in Figure \ref{n-strand-paren} correspond to
	\begin{align*}
		\lambda(a \Omega^1) \Delta^{(n)}_R(\Omega^2).
	\end{align*}
	\begin{figure}[ht]
		\labellist
		\pinlabel $(n)$ at 113 315
		\pinlabel $_R$ at 135 250
		\pinlabel $_R$ at 135 225
		\pinlabel $_R$ at 135 25
		\pinlabel $a$ at -10 145
		\endlabellist
		\centering
		\includegraphics[scale=0.6]{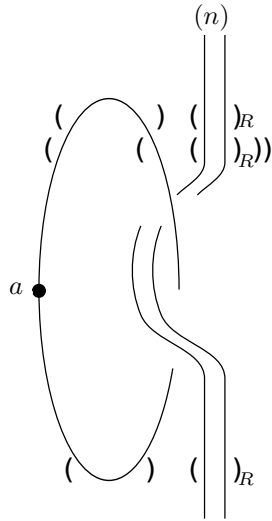}
		\caption{Set Up for the Fenn-Rourke Move}
		\label{n-strand-paren}
	\end{figure}
\end{lemma}

\begin{remark}
	This hybrid-type diagram in Figure \ref{n-strand-paren} indicates that the bead $a$ will be applied under $\Psi$.  When we apply the lemma, the bead $a$ will come from the framing of the closed component, and hence will satisfy the conditions of Lemma \ref{n-strand-FR-helper}.
\end{remark}

\begin{proof}
		The tangle resulting from applying the $n$-fold coproduct to the vertical component in Lemma \ref{one-strand-FR} is pictured in Figure \ref{n-strand-paren}.  Under our functor $\Psi$, for the upper crossing in Figure \ref{n-strand-paren} we apply the factors of $(\id \otimes \Delta^{(n)}_R)(\Rmat)$ to the $n+1$ strands and for the lower crossing we apply the factors of $(\Delta^{(n)}_R \otimes \id)(\Rmat)$ to the $n+1$ strands.
		
		To simplify the figure, instead of our usual bead decorations, we will decorate with horizontal lines labeled with the element whose factors would be used for the beads at that level.  See Figure \ref{n-strand-beads}.
	
	\begin{figure}[ht]
        \hspace{-1in}
		\labellist
		\pinlabel $(n)$ at 114 315
		\pinlabel $_R$ at 135 240
		\pinlabel $_R$ at 134 220
		\pinlabel $_R$ at 135 25
		\pinlabel $a$ at 20 145
        \pinlabel $\Psi$ at 180 175
		\pinlabel $(n)$ at 357 315
		\pinlabel $\beta$ at 263 260
		\pinlabel $a$ at 235 145
		\pinlabel $\alpha$ at 317 30
		\pinlabel {$(1\otimes 1 \otimes \Delta^{(n)}_R)(\Phi\inv)$} at 475 243
		\pinlabel {$1 \otimes (1 \otimes \Delta^{(n)}_R)({\mathcal R}) $} at 465 192
		\pinlabel {$1 \otimes (1 \otimes \Delta^{(n)}_R)({\mathcal R}^{\rm op})$} at 475 110
		\pinlabel {$(1 \otimes 1 \otimes \Delta^{(n)}_R)(\Phi)$} at 465 40
		\endlabellist
		\includegraphics[scale=0.5]{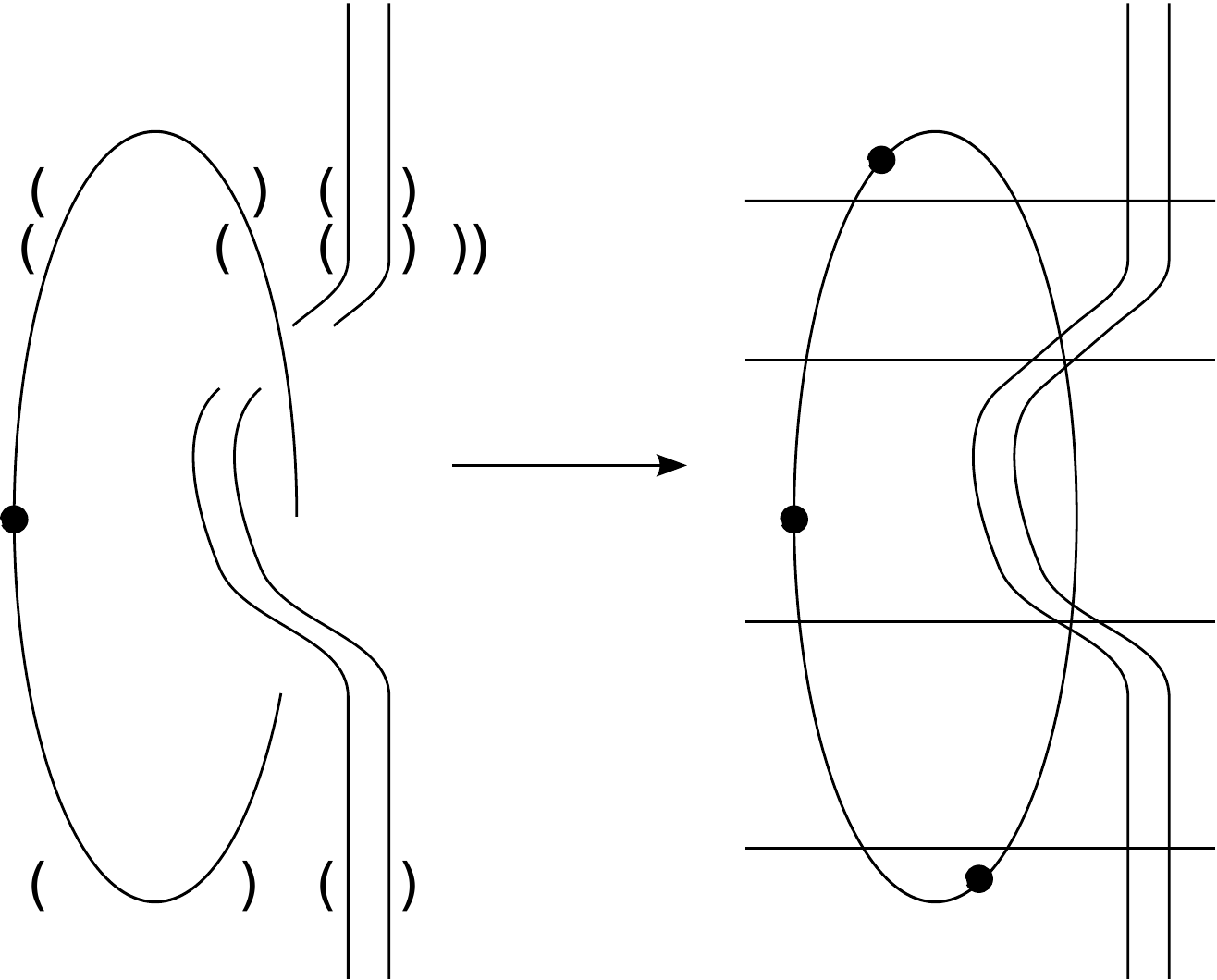}
		\caption{The Functor $\Psi$ Applied to the Tangle in Figure \ref{n-strand-paren}}
		\label{n-strand-beads}
	\end{figure}
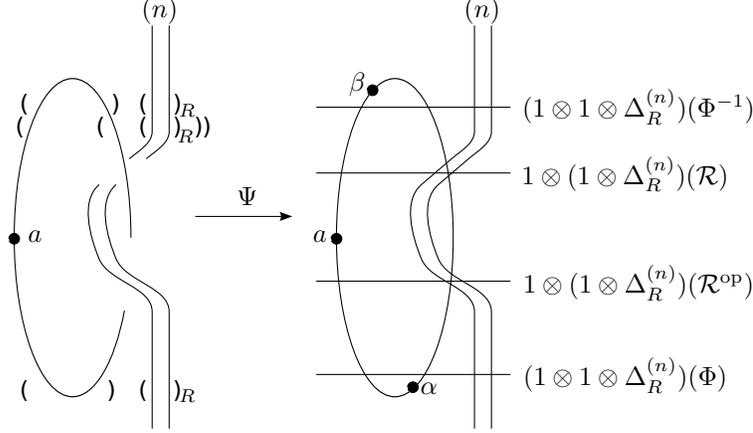
	
	We collect the beads on the vertical strands and on the right-hand side of the circular component.  First, note that we can simplify the calculations at the top and bottom of the diagram.  We move the $\beta$ across the maximum and combine it with the beads directly below, and then combine the $\alpha$ with the beads directly above.
	\begin{align}
		&(1 \otimes S\inv(\beta) \otimes 1 \otimes ... \otimes 1)(S\inv \otimes 1 \otimes \Delta^{(n)}_R)(\Phi\inv) = 1 \otimes (1 \otimes \Delta^{(n)}_R)p_L \nonumber \\
		&(S \otimes 1 \otimes \Delta^{(n)}_R)(\Phi) (1 \otimes \alpha \otimes 1 \otimes ... \otimes 1) = 1 \otimes (1 \otimes \Delta^{(n)}_R)q_L
	\end{align}
	The bead $a$ may be placed anywhere in the diagram as we have assumed it is in the center of $H$ and $S(a) = a$.  The two horizontal lines directly above and below the bead $a$ in Figure \ref{n-strand-beads} represent two applications of $\Rmat$ on opposite strands, hence may be combined to give the element
	\begin{align*}
		1 \otimes (1 \otimes \Delta^{(n)}_R)({\cal M}).
	\end{align*}
	Note that the left-most coordinates in all of these elements are all the identity in $H$.  We combine all of this information together to see that we have
	\begin{align}
		&(1 \otimes \Delta^{(n)}_R)q_L \cdot (1 \otimes \Delta^{(n)}_R)({\cal M}) \cdot (1 \otimes \Delta^{(n)}_R)p_L \nonumber \\
		&= (1 \otimes \Delta^{(n)}_R)(q_L {\cal M} p_L). \label{n-strand-fr-beads}
	\end{align}
	
	By the definition of $\Omega$, we rewrite \eqref{n-strand-fr-beads} as
	\begin{align*}
		(1 \otimes \Delta^{(n)}_R) (\Omega).
	\end{align*}
	Multiplying by $a$ and applying the cointegral $\lambda$ on the first factor gives our final result in $Vect$:
	\begin{align*}
		\lambda(a \Omega^1) \Delta^{(n)}_R(\Omega^2),
	\end{align*}
	which is the desired result.

\end{proof}

\begin{lemma} \label{n-strand-FR}
	The functor $\Xi$ factors through the Fenn-Rourke move.
\end{lemma}

\begin{proof}
	This is the result of the cases $a = v$ and $a = v\inv$ in Lemma \ref{n-strand-FR-helper}.
	
	A framing of $-1$ on the circular component of Figure \ref{n-strand-paren} corresponds to taking $a = v\inv$.  Lemma \ref{n-strand-FR-helper} gives that the strands are decorated with the element 
	\begin{align*}
		\lambda(v\inv \Omega^1) \Delta^{(n)}_R(\Omega^2).
	\end{align*}
	In Lemma \ref{one-strand-FR}, we proved that 
	\begin{align*}
		\lambda(v\inv \Omega^1) \Omega^2 = v \cdot \lambda(\alpha v\inv S\inv(\beta)).
	\end{align*}
	Applying the $n$-fold coproduct to both sides of this equation gives
	\begin{align}
		\lambda(v\inv \Omega^1) \Delta^{(n)}_R(\Omega^2) = \Delta^{(n)}_R(v) \lambda(\alpha v\inv S\inv(\beta)). \label{frfinal}
	\end{align}
	The element on the left-hand side of \ref{frfinal} is the result of applying $\hat{\Psi} \circ \Psi$ to the tangle on the left-hand side of Figure \ref{frproof}, while the element on the right-hand side of \ref{frfinal} is the result of applying $\hat{\Psi} \circ \Psi$ to the tangle on the right-hand side of Figure \ref{frproof}.  That is, the map $\tilde{\nu}_H$ factors through the Fenn-Rourke move in the case of $a = v\inv$.
	\begin{figure}[ht]
		\centering
		\labellist
		\pinlabel $(n)$ at 110 315
		\pinlabel $_R$ at 132 250
		\pinlabel $_R$ at 132 225
		\pinlabel $_R$ at 132 25
		\pinlabel $-1$ at -12 145
		\pinlabel $-1$ at 445 160
		\pinlabel $(n)$ at 250 260
		\pinlabel $_R$ at 270 232
		\endlabellist
		\includegraphics[scale=0.7]{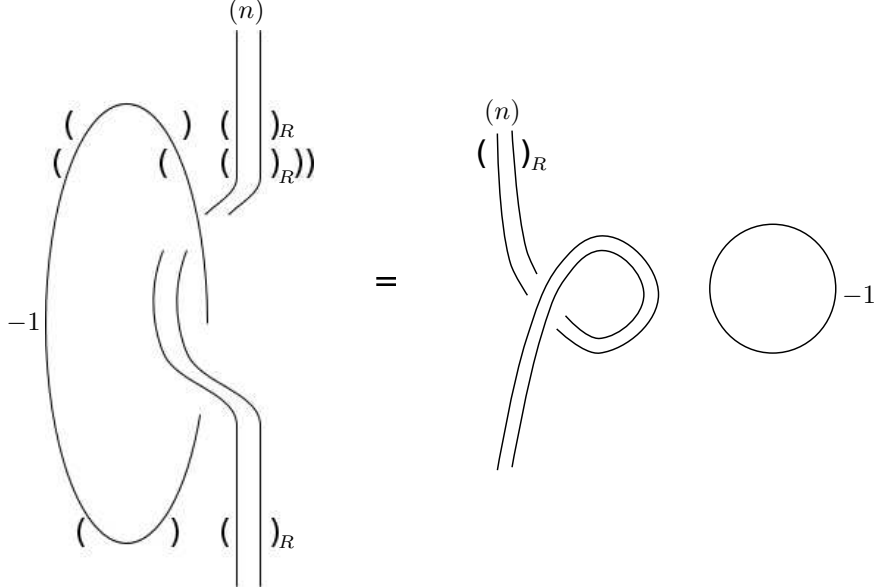}
		\caption{The Fenn-Rourke Move}
		\label{frproof}
	\end{figure}
	To complete the proof, we must also consider the case of a $+1$ framing, or $a = v$.  We reverse all of the crossings.  Lemma \ref{one-strand-FR} becomes 
	\begin{align*}
		\lambda(a \tilde{\Omega}^1) \tilde{\Omega}^2 = v\inv \lambda(\alpha v S\inv \beta),
	\end{align*}
	where
	\begin{align*}
		\tilde{\Omega} = \tilde{\Omega}^1 \otimes \tilde{\Omega}^2 = q_L {\cal M}\inv p_L.
	\end{align*} 
	Applying $\Delta^{(n)}_R$ to both sides of this equation gives 
	\begin{align*}
		\lambda(v \tilde{\Omega}^1) \Delta^{(n)}_R(\tilde{\Omega}^2) = \Delta^{(n)}_R(v\inv) \lambda(\alpha v S\inv(\beta)),
	\end{align*}
	which is the desired result if we reverse the crossings in Figure \ref{frproof}.  Hence, our functor $\Xi$ factors through the Fenn-Rourke move, as desired.

\end{proof}

The other case of Lemma \ref{n-strand-FR-helper} with which we are concerned is the case $a = 1$, and this gives us a formula for an integral $\Lambda$.
\begin{lemma} \label{wegetanintegral}
	The following gives a formula for an integral $\Lambda \in H$.
	\begin{equation}\label{integraleqn}
		\Lambda = \lambda(\alpha S\inv( \beta)) (\lambda \otimes 1) (q_L {\cal M} p_L) 
	\end{equation}
\end{lemma}

\begin{proof}
	We have just shown in Lemma \ref{n-strand-FR} that our functor factors through the Fenn-Rourke move.  Using the fact that the Fenn-Rourke Move is equivalent to the Second Kirby Move, we have the equivalence of the tangles in Figure \ref{integralproof}.
	\begin{figure}[ht]
		\centering
		\includegraphics[scale=0.6]{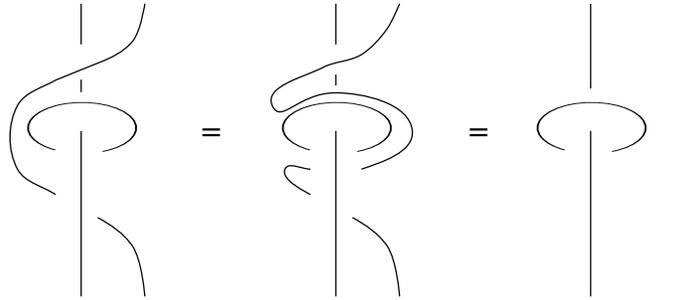}
		\caption{The Tangle of Figure \ref{integraltangle} Gives an Integral}
		\label{integralproof}
	\end{figure}
	The assumption that ${\cal M}$ is non-degenerate is equivalent to the statement that the tangle on the left-hand side of Figure \ref{integraltangle} must map to a straight strand labeled with an integral under $\Psi$.  See \cite{CK}, where the result is given for ordinary Hopf algebras.  The proof is the same in the quasi-Hopf case.
	\begin{figure}[ht]
			\centering
			\labellist
			\pinlabel $0$ at 60 140
			\pinlabel $\Lambda$ at 335 100
			\endlabellist
			\includegraphics[scale=0.6]{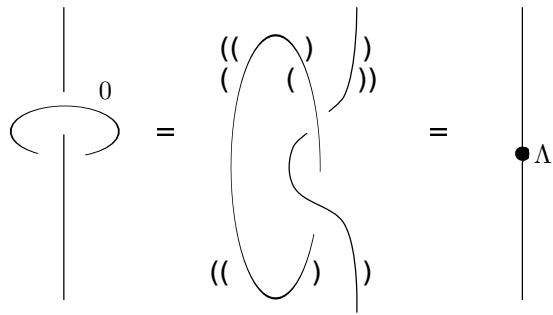}
			\caption{A Tangle for the Integral}
			\label{integraltangle}
	\end{figure}
	Applying the functor $\hat{\Psi} \circ \Psi$ to the tangle on the left-hand side of Figure \ref{integraltangle} gives the algebraic formula for $\Lambda$ in \eqref{integraleqn}.

\end{proof}

We use Lemma \ref{wegetanintegral} to prove the final lemma.

\begin{lemma} \label{sigmamovelemma}
    The functor $\Xi$ factors through the $\sigma$ Move, depicted in Figure \ref{sigmamove}.
	\begin{figure}[ht]
		\centering
		\includegraphics[scale=0.5]{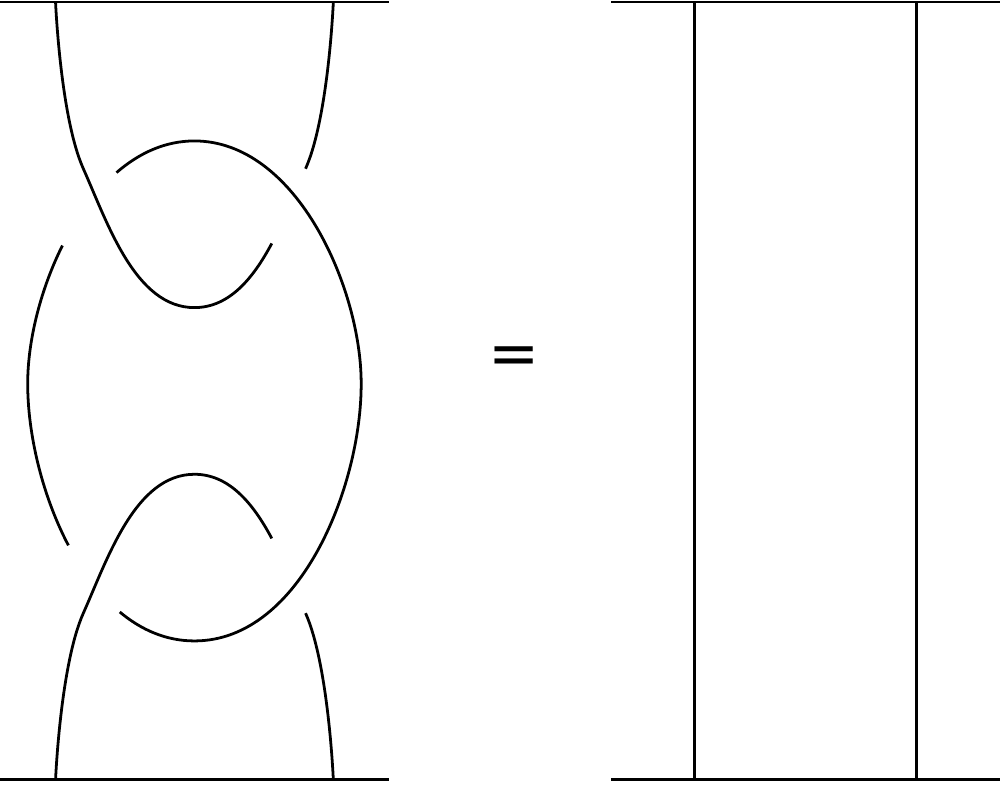}
		\caption{The $\sigma$ Move}
		\label{sigmamove}
	\end{figure}
\end{lemma}

\begin{proof}
	Note that the diagram for the $\sigma$-move should be properly parenthesized, which we will do shortly.  See Figure \ref{sigmabeads} for the beaded diagram produced by $\Psi$.  We have used several isotopies.  First, we move the minimum and the maximum into side-by-side positions, then stretch the minimum downward and the maximum upward.  Next, we shrink the center circle so that it surrounds the inner strands.  Finally, we change the parenthesis to simplify our calculations.  
	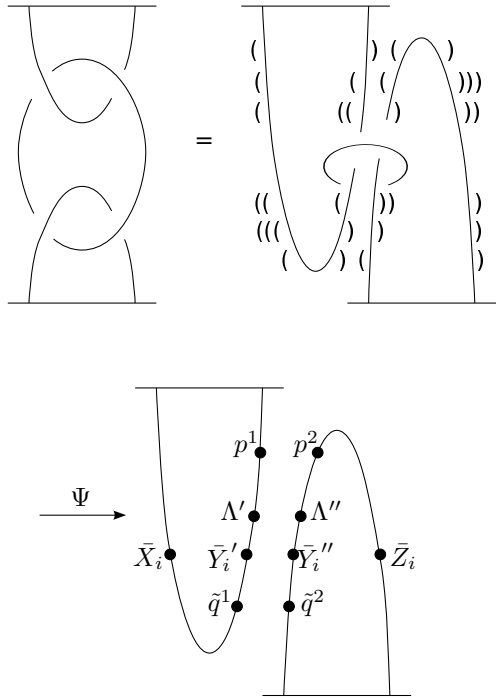
\begin{figure}[ht]
		\centering
		\labellist
        \pinlabel $\Psi$ at 55 150
		\pinlabel $p^1$ at 180 190
		\pinlabel $p^2$ at 225 190
		\pinlabel $\Lambda'$ at 170 140
		\pinlabel $\Lambda''$ at 240 140
		\pinlabel $\bar{X_i}$ at 106 105
		\pinlabel $\bar{Y_i}'$ at 162 105
		\pinlabel $\bar{Y_i}''$ at 232 105
		\pinlabel $\bar{Z_i}$ at 298 105
		\pinlabel $\tilde{q}^1$ at 160 70
		\pinlabel $\tilde{q}^2$ at 230 70
		\endlabellist
		\includegraphics[scale=0.5]{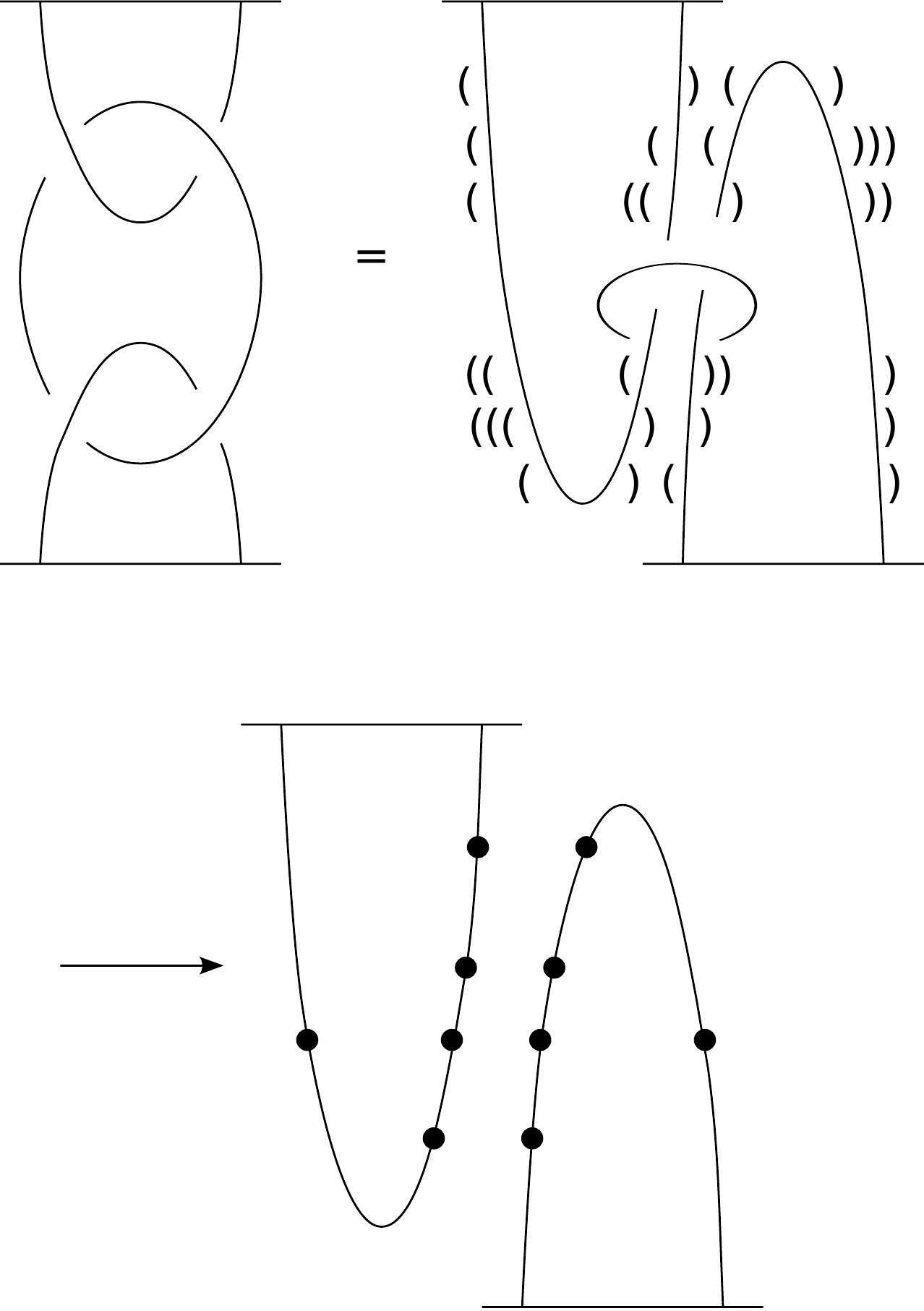}
		\caption{The Map $\tilde{\nu}_H$ Applied to the $\sigma$-Move}
		\label{sigmabeads}
	\end{figure}
	
	Collecting the beads and applying $\hat{\Psi}$ to this diagram, we have the following map.
	\begin{equation} \label{sigma1}
		x \mapsto \lambda(S(x) S(\bar{X_i})\tilde{q}^1 \bar{Y_i}' \Lambda' p^1) \tilde{q}^2 \bar{Y_i}'' \Lambda'' p^2 S(\bar{Z_i})
	\end{equation}
	One property of the elements $p_R$ and $p_L$ is that we have
	\begin{equation*}
		\Lambda' p^1 \otimes \Lambda'' p^2 = \Lambda' \tilde{p}^1 \otimes \Lambda'' \tilde{p}^2
	\end{equation*}
	for any right integral $\Lambda$; see \cite{BC}, equation (3.21).  We substitute this into \eqref{sigma1}.
	\begin{equation}\label{sigma2}
		x \mapsto \lambda(S(x) S(\bar{X_i})\tilde{q}^1 \bar{Y_i}' \Lambda' \tilde{p}^1) \tilde{q}^2 \bar{Y_i}'' \Lambda'' \tilde{p}^2 S(\bar{Z_i})
	\end{equation}
	Next, we note that the coassociator here is multiplied by $\Delta(\Lambda)$ as follows.
		\begin{equation*}
			\bar{X_i} \otimes \bar{Y_i }' \Lambda' \otimes \bar{ Y_i }'' \Lambda'' \otimes \bar{Z_i} = (1 \otimes \Delta \otimes 1)(\Phi\inv) \cdot (1 \otimes \Delta(\Lambda) \otimes 1)
		\end{equation*} 
	This can be rewritten as 
		\begin{equation*}
			(1 \otimes \Delta \otimes 1)(\Phi\inv \cdot (1 \otimes \Lambda \otimes 1)) = \bar{X_i} \otimes \Delta(\bar{Y_i} \Lambda) \otimes \bar{Z_i}.
		\end{equation*}
		 Since $\Lambda$ is also a left integral, however, we know that $\bar{Y_i} \Lambda = \epsilon(\bar{Y_i}) \Lambda$, so our expression becomes
		 \begin{equation*}
		 	\bar{X_i} \otimes \epsilon(\bar{Y_i}) \Delta(\Lambda) \otimes \bar{Z_i} = 1 \otimes \Delta(\Lambda) \otimes 1,
		 \end{equation*}
	since $(1 \otimes \epsilon \otimes 1)(\Phi\inv) = 1 \otimes 1 \otimes 1$.
	
	Substituting this result into \eqref{sigma2}, we have the following.
	\begin{align}
		x &\mapsto \lambda(S(x) \tilde{q}^1 \Lambda' \tilde{p}^1) \tilde{q}^2 \Lambda'' \tilde{p}^2 \label{sigma2-5} \\
		&= \tilde{q}^2 \lambda(S(x) \tilde{q}^1 \Lambda' \tilde{p}^1) \Lambda'' \tilde{p}^2 \label{sigma3}
	\end{align}
	Finally, we use \eqref{BCcond3} with $h = S(x) \tilde{q}^1$ to simplify \eqref{sigma3}.  This gives the final simplification of \eqref{sigma1}.
	\begin{align}
		x &\mapsto  \tilde{q}^2 \epsilon(\beta) \lambda(\Lambda) S\inv(S(x)\tilde{q}^1 \beta)  \nonumber \\
		&= \lambda(\Lambda) \tilde{q}^2 S\inv (\beta) S\inv(\tilde{q}^1) x \nonumber\\
		&= \lambda(\Lambda) \bar{Z_j} S\inv(\beta) S\inv(\bar{Y_j}) S\inv(\alpha) \bar{X_j} x \label{subqdef}\\
		&= \lambda(\Lambda) S\inv(S(\bar{X_j}) \alpha \bar{Y_j} \beta S(\bar{Z_j})) x \nonumber \\
		&= \lambda(\Lambda) S\inv(1)x \nonumber\\
		&= \lambda(\Lambda) x = x \nonumber
	\end{align}
	Note that we have assumed $\lambda(\Lambda) = 1$.  Also, in \eqref{subqdef} we have substituted the definition of $q_L$, given originally in \eqref{qldef}.  Hence the relation on tangles in Figure \ref{sigmamove} holds algebraically under the functor $\tilde{\nu}_H$, as desired.
\end{proof}

These are all of the required lemmas, so the proof of Theorem \ref{thisisthemainidea} is complete.


\end{document}